\documentclass[12pt]{amsart}

\usepackage{amssymb,amsmath,latexsym}
\usepackage{amsthm}
\usepackage{amsfonts}
\usepackage{enumerate}
\usepackage{booktabs}
\usepackage{mathrsfs}

\usepackage{hyperref}

\usepackage{graphicx,pgf}
\usepackage{float}

\usepackage{caption}
\usepackage{subcaption}

\usepackage{fancyhdr}

\theoremstyle{plain} 
\newtheorem{thm}{Theorem}[section]
\newtheorem{prop}[thm]{Proposition}
\newtheorem{lemma}[thm]{Lemma}
\newtheorem{cor}[thm]{Corollary}
\newtheorem{example}[thm]{Example}
\newtheorem{rmk}{Remark}

\theoremstyle{definition}
\newtheorem{defn}[thm]{Definition}

\theoremstyle{remark}

\numberwithin{equation}{section}

\newcommand{\mysectionname}{}
\newcommand{\newsection}[1]{\section{#1}\renewcommand{\mysectionname}{\uppercase{#1}}}

\newcommand{\MT}{\mathcal{M}_{\mathrm{T}}}  
\newcommand{\Mstar}{\mathcal{M}_*}
\newcommand{\MstarS}{\mathcal{M}_{\mathrm{T}}^*}
\newcommand{\Mreal}{\mathcal{M}_{\mathbb{R}}}
\newcommand{\Mrealplus}{\mathcal{M}_{\mathbb{R}^+}}
\newcommand{\Mrealpplus}{\mathcal{M}_{\mathbb{R}^+}^*}

\newcommand{\g}[1]{\mathit{G}_{#1}}
\newcommand{\f}[1]{\mathit{F}_{#1}}
\newcommand{\e}[1]{\mathit{E}_{#1}}
\newcommand{\p}[1]{\varphi_{#1}}
\newcommand{\pp}[1]{\psi_{#1}}
\newcommand{\et}[1]{\eta_{#1} }
\newcommand{\multt}[1]{{#1}^{\boxtimes t }}
\newcommand{\multtT}[1]{{\mu}^{\boxtimes #1 }}

\newcommand{\sig}[1]{\Sigma_{#1}}

\newcommand{\bmultg}{\!\begin{array}{c} {\scriptstyle\times} \\[-12pt]\cup\end{array}\!}
\newcommand{\bmultk}{\!\!{\scriptstyle\begin{array}{c} {\scriptscriptstyle\times} \\[-12pt]\cup\end{array}}\!\!}

\newcommand{\subord}{\ \frame{$\ \vdash$}\; }  

\begin{document}
\title{Free Brownian motion and free convolution semigroups: multiplicative case}
\author{Ping Zhong}
\address{Department of Mathematics, Rawles Hall, 831 East Third Street, Indiana University, Bloomington, Indiana 47405, U.S.A. }
\email{pzhong@indiana.edu}
\begin{abstract}
We consider a pair of probability measures $\mu,\nu$ on the unit circle
such that $\Sigma_{\lambda}(\eta_{\nu}(z))=z/\eta_{\mu}(z)$.
We prove that the same type of equation holds for any $t\geq 0$ when we replace
$\nu$ by $\nu\boxtimes\lambda_t$ and
$\mu$ by $\mathbb{M}_t(\mu)$, where
$\lambda_t$ is the free multiplicative analogue of the normal distribution on the unit circle of $\mathbb{C}$ and
$\mathbb{M}_t$ is the map defined by Arizmendi and Hasebe.
These equations are a multiplicative analogue of
 equations studied by Belinschi and Nica.
In order to achieve this result, we study
infinite divisibility of the measures associated with subordination functions
in multiplicative free Brownian motion and multiplicative
free convolution semigroups. We use the modified $\mathcal{S}$-transform introduced by Raj Rao and Speicher to deal with the case that $\nu$ has mean zero. The same type of the result holds for
convolutions on the positive real line. In the end,
we give a new proof for some Biane's results on the densities of
the free multiplicative analogue of the normal distributions. 
\end{abstract}

\maketitle
\newsection{Introduction}
Let $\Mreal$ be the set of probability measures on $\mathbb{R}$.
For every $t\geq 0$, Belinschi and Nica \cite{BN2008Ind} defined
a family of maps $\mathbb{B}_t : \Mreal \rightarrow\Mreal$ by setting
\begin{equation}\nonumber
   \mathbb{B}_t(\mu)=\left(\mu^{\boxplus (t+1)} \right)^{\uplus \frac{1}{t+1}}, \,\,\mu\in\Mreal.
\end{equation}
These maps have several remarkable properties.
For any $t\geq 0$, $\mathbb{B}_t$ is an endomorphism
of $(\Mrealplus,\boxtimes)$,
where $\Mrealplus $ is the set of probability measures on $[0,+\infty)$
and $\boxtimes$ is free multiplicative convolution. $\{\mathbb{B}_t\}_{t\geq 0}$ is a semigroup and $\mathbb{B}_1$ is the Boolean to free Bercovici-Pata bijective map.

The maps $\mathbb{B}_t$ have strong connections with $\boxplus$-infinite divisibility.
They are also connected to free Brownian motion and additive free
convolution semigroups. For $\mu\in\Mreal$, we denote by $\g{\mu}$
the Cauchy transform of $\mu$ and by $\f{\mu}$
the reciprocal Cauchy transform of $\mu$.
Given a pair of probability measures $\mu,\nu \in\Mreal$
such that
\begin{equation}\nonumber
  \g{\nu}(z)=z-\f{\mu}(z),\,\,\,z\in\mathbb{C}^+,
\end{equation}
we have
\begin{equation}\label{eq:int1}
   \g{\nu\boxplus\gamma_t}(z)=z-\f{\mathbb{B}_t(\mu)}(z), \,\,\, t>0, \,\,\,z\in\mathbb{C}^+,
\end{equation}
where $\gamma_t$ is the semi-circular distribution with variance $t$.
This result was generalized to the
multi-variable case in \cite{BN2008,BN2009,Nica2009}.
An equivalent form of (\ref{eq:int1}) was used to 
the superconvergence theorem in \cite{Wang2010}. In a series of papers 
\cite{Anshelevich,AnshG2011,AnshetwostateBM,Anshelevich2012},
Anshelevich generalized the above correspondence of $\mu \leftrightarrow\nu$
and $\mathbb{B}_t(\mu)\leftrightarrow\nu\boxplus\gamma_t$
to the context of two-state probability spaces.
Motivated by these generalizations and applications, we study in this article
the analogue of these equations
for multiplicative free convolution.

Throughout this article, we denote by $\mathrm{T}$ the unit circle of $\mathbb{C}$, by $\MT$ the set of probability measures on $\mathrm{T}$,
and by $\Mstar$ the set of probability measures on $\mathbb{C}$ with nonzero
mean.
We also set
\begin{equation}\nonumber
  \mathcal{M}_{\mathrm{T}}^*=\{\mu\in\MT\cap \Mstar:
  \eta_{\mu}(z)\neq 0, \forall z\in\mathbb{D}\backslash \{0\} \}.
\end{equation}
It was shown in \cite{BB2005} that one can define multiplicative free
convolution power $\multt{\mu}$ for $\mu\in\MstarS$ and $t>1$.

In \cite{AHasebe2012},
a family of maps $\mathbb{M}_t$, which is the analogue of the semigroup $\mathbb{B}_t$,
was defined for the probability measures in $\MstarS$.
The definition of $\mathbb{M}_t$ in \cite{AHasebe2012} is more general, and
we only need a simpler form defined as follows. Given $\mu\in\MstarS$ which has positive mean,
then for $t\geq 0$, the map $\mathbb{M}_t$ is defined by
\begin{equation}\nonumber
   \mathbb{M}_t(\mu)=\left(\mu^{\boxtimes (t+1)} \right)^{\bmultk \frac{1}{t+1}},
\end{equation}
where the convolution power $\mu^{\boxtimes (t+1)}$ and the measure $\mathbb{M}_t(\mu)$ are chosen in a way such that they have positive means.

We then state one of our main theorems.
\begin{thm}\label{thm:1.1}
  Given a pair of probability measures $\mu\in \MstarS$ and $\nu\in\MT$
  such that
     \begin{equation}\label{equalM}
       \Sigma_{\lambda}(\eta_{\nu}(z))=\frac{z}{\eta_{\mu}(z)}, z\in \mathbb{D},
     \end{equation}
    we have
     \begin{equation}
       \Sigma_{\lambda}(\eta_{\nu\boxtimes\lambda_t}(z))=
       \frac{z}{\eta_{\mathbb{M}_t(\mu)}(z)}, z\in \mathbb{D},
     \end{equation}
     where $\lambda_t$ is the analogue of the normal distribution on $\mathit{T}$ with $\Sigma_{\lambda_t}(z)=\exp(\frac{t}{2}\frac{1+z}{1-z})$ and $\lambda=\lambda_1$.
\end{thm}
In order to prove Theorem \ref{thm:1.1}, we consider two semigroups
$\nu\boxtimes\lambda_t$ and
$\mu^{\boxtimes (t+1)}$  for all $t\geq0$.
It is well-known that
$\et{\nu\boxtimes\lambda_t}$ and $\et{\mu^{\boxtimes (t+1)}}$ are subordinated to $\et{\nu}$ and $\et{\mu}$ respectively.
We prove that the subordination functions are $\eta$-transforms of some
$\boxtimes$-infinitely divisible probability measures on $\mathrm{T}$.
It turns out that the equation $\Sigma_{\lambda}(\eta_{\nu}(z))=z/\eta_{\mu}(z)$ means that the subordination function of $\et{\nu\boxtimes\lambda_t}$ with respect to $\et{\nu}$ and the subordination function
of $\et{\mu^{\boxtimes (t+1)}}$ with respect to $\et{\mu}$  are the same. The proof of Theorem \ref{thm:1.1}
will be given in Subsection 3.5.

Given $\mu\in\MT$, we prove that if $\multt{\mu}$ can be defined
and $\et{\multt{\mu}}$ is subordinated to $\et{\mu}$ for
all $t>1$, then $\mu\in\MstarS$; in addition, we prove that
for nontrivial measures $\mu\in\MT$ and $\nu\in\mathcal{ID}(\boxtimes, \mathrm{T})$, the density functions of the measures $\mu\boxtimes\nu_t$ and $\mu^{\boxtimes t}$
 converge to $1/2\pi$ uniformly as $t\rightarrow \infty$.

To deal with the case that $\nu\in\MT\backslash\Mstar$, we use the modified $\mathcal{S}$-transform  \cite{Ariz2012, RS2007}
to study subordination functions. In this case, the subordination function
of $\et{\nu\boxtimes\lambda_t}$ with respect to $\et{\nu}$ is generally not unique.
However, we can prove that there exists a unique subordination function satisfying certain properties (see Theorem \ref{BBnew}). Let $\rho_t$ be the measure associated with this subordination function
of $\et{\nu\boxtimes\lambda_t}$ with respect to $\et{\nu}$, we have 
that $\Sigma_{\rho_t}(z)=\Sigma_{\lambda_t}(\et{\nu}(z))$.

Similar results to Theorem \ref{thm:1.1} for
multiplicative convolution on $\Mrealplus$ are also valid. The proof for this case is much simpler because of the uniqueness of multiplicative convolution powers and the uniqueness of subordination functions.

Finally, we give a new proof for some results concerning the density functions of the free multiplicative 
analogue of the normal distributions studied by Biane in \cite{BianeJFA}, and we obtain
some new results. For example,
for $\lambda_t$ ($t> 0$) the free multiplicative analogue of the normal distributions on $\mathrm{T}$,
we prove that $\lambda_t$ is unimodal.

This article is organized as follows. After this introductory section, we describe some backgrounds in the additive case
in Section 2. In Section 3, we consider multiplicative free and multiplicative Boolean convolution on $\MT$, and
prove our main theorems.
Section 4 is devoted to studying multiplicative free and multiplicative Boolean convolution on $\Mrealplus$. The regularity properties of the free multiplicative analogue of the normal distributions
are discussed in Section 5.
\newsection{Background: additive case}
\subsection{Additive free convolution and additive Boolean convolution}
For a measure $\mu\in\Mreal$,
we define the Cauchy transform $\g{\mu} : \mathbb{C}^+ \rightarrow
\mathbb{C}^-$ by
\begin{equation}\nonumber
  \g{\mu}(z)=\int_{-\infty}^{+\infty}\frac{1}{z-t}\,d\mu(t),\,\,\,z\in\mathbb{C}^+.
\end{equation}
We set $\f{\mu}(z)=1/\g{\mu}(z), \, z\in\mathbb{C}^+$, so that $\f{\mu}:
\mathbb{C}^+\rightarrow\mathbb{C}^+$ is analytic.

The following result in \cite{BV1993} characterizes those functions which are
reciprocal Cauchy transforms of probability measures.
\begin{prop}\label{fmu}
Let $\mathit{F}:\mathbb{C}^+\rightarrow\mathbb{C}^+$ be an analytic function.
Then the following assertions are equivalent.
\begin{enumerate}
  \item There exists a probability measure $\mu$ on $\mathbb{R}$ such that
  $\mathit{F}(z)=\f{\mu}(z)$ in $\mathbb{C}^+$.
  \item There exists $a\in \mathbb{R}$, and a finite positive measure $\rho $ on $\mathbb{R}$
  such that
   \begin{equation}\nonumber
     \mathit{F}(z)=a+z+\int_{-\infty}^{+\infty}\frac{1+tz}{t-z}\,d\rho(t)
   \end{equation}
  for all $z\in \mathbb{C}^+$.
  \item We have that $\lim_{y\rightarrow +\infty}\f{}(iy)/iy=1$.
\end{enumerate}
\end{prop}

It was proved in \cite{BV1993} that $\f{\mu}$ is invertible in some domain. More precisely,
for two positive numbers $M$ and $N$, we set
\begin{equation}\nonumber
    \Gamma_{M,N}=\{z\in\mathbb{C}^+ : |x|<M y,\, y>N \}.
\end{equation}
Then for any $M>0$, there exists $N>0$ such that the left inverse $\f{\mu}^{-1}$ of $\f{\mu}$
is defined in $\Gamma_{M,N}$, and then we can define the Voiculescu transform of $\mu$
by
\begin{equation}\nonumber
  \p{\mu}(z)=\f{\mu}^{-1}(z)-z,
\end{equation}
for $z\in \Gamma_{M,N}$. For any two measures $\mu, \nu \in \Mreal$, we have that
\begin{equation}\label{sumremark}
  \p{\mu\boxplus\nu}(z)=\p{\mu}(z)+\p{\nu}(z)
\end{equation}
holds in any truncated cone $\Gamma_{M,N}$
where $\p{\mu},\p{\nu}$ and $\p{\mu\boxplus\nu}$ are defined. This remarkable result
was discovered by Voiculescu \cite{DVV1986} for compactly supported measures and then
extended to general cases in \cite{BV1993,Maassen}.

Given $\nu\in\Mreal$, we say that $\nu$ is $\boxplus$-infinitely divisible if for every positive
integer $n$, there exists a probability measure $\nu_{1/n}\in \Mreal$ such that
\begin{equation}\nonumber
 \nu=\underbrace{\nu_{1/n}\boxplus\nu_{1/n}\boxplus\cdots\boxplus\nu_{1/n}}_{n
 \,\,\text{times}}.
\end{equation}
It is known \cite{BV1993, Maassen, DVV1986} that a probability measure $\nu$ on $\mathbb{R}$ is $\boxplus$-infinitely divisible if
 and only if its Voiculescu transform $\p{\nu}$ has an analytic extension defined on
 $\mathbb{C}^+$ with values in $\mathbb{C}^-\cup\mathbb{R}$.
 We denote by $\mathcal{ID}(\boxplus, \mathbb{R})$ the set of all
$\boxplus$-infinitely divisible probability measures on the real line.
If $\nu\in\mathcal{ID}(\boxplus, \mathbb{R})$, then for every $t>0$, there exists a probability
measure $\nu_t$ such that $\varphi_{\nu_t}(z)=t\varphi_{\nu}(z)$ for $z$ in the common domain of $\varphi_{\nu}$
and $\varphi_{\nu_t}$.

We would like to mention the following fact.
\begin{prop}\label{leftInvreal}
 If $\nu$ is $\boxplus$-infinitely divisible, let $\mathit{H}(z)=z+\varphi_{\nu}(z)$, then
 \begin{equation}\label{leftInvFor}
   \mathit{H}(\f{\nu}(z))=z
 \end{equation}
 holds for $z \in \mathbb{C}^+$. The set $U:=\{z\in\mathbb{C}^+: \Im{\mathit{H}(z)}>0 \}$ is
 a simply connected domain with boundary which is a simple curve
 and $\mathit{H}$ maps $\mathbb{C}^+$ conformally onto $U$. Moreover, the boundary $\partial U$ is
 the graph of a function and the function $\mathit{H}$ is continuous up to the real axis.
\end{prop}
\begin{proof}
 The first part of the assertion appears in \cite{BV1993,DVV1986},
 and the second part of the assertion follows from the fact that $\mathit{H}$ satisfies
 the conditions of Proposition 4.7 in \cite{BB2005}. The last part of the assertion is due to Lemma 3.3 in \cite{CG2011}
 and Proposition 4.7 in \cite{BB2005}.
\end{proof}

Additive Boolean convolution was introduced in \cite{SpeicherW1997}. For $\mu \in \Mreal$,
we set
$\e{\mu}(z)=z-\f{\mu}(z)$.
For $\mu,\nu\in\Mreal$, the additive Boolean convolution $\mu\uplus\nu$ is characterized by the identity
\begin{equation}\nonumber
  \e{\rho}(z)=\e{\mu}(z)+\e{\nu}(z), \, \text{for}\, z\in\mathbb{C}^+.
\end{equation}
We can also consider the infinite divisibility with respect to additive Boolean convolution.
It turns out that every $\mu\in\Mreal$ is $\uplus$-infinitely divisible (see \cite{SpeicherW1997}).
We denote by $\mathcal{ID}(\uplus, \mathbb{R})$ the set of all
$\uplus$-infinitely divisible probability measures on the real line.
\subsection{Infinite divisibility and subordination functions}\label{andand}
Given $\mu, \nu \in\Mreal$, it is known that $\f{\mu\boxplus\nu}$ is subordinated to $\f{\mu}$ and $\f{\nu}$,
and by Proposition \ref{fmu}, we can also regard these subordination functions as the reciprocal Cauchy transforms
of probability measures on $\mathbb{R}$.
\begin{defn}
 For $\mu, \nu\in\Mreal$, the subordination distribution \cite{Anshelevich2012, LR2007, Nica2009}
 $\mu\small{\subord}\nu$ (resp. $\nu\small{\subord}\mu$) is defined to be the unique probability measure in $\Mreal$
 such that $\f{\mu\boxplus\nu}(z)=\f{\nu}(\f{\mu\small{\subord}\nu}(z))$ (resp. $\f{\mu\boxplus\nu}(z)=\f{\mu}(\f{\nu\small{\subord}\mu}(z))$).
\end{defn}

Many subordination distributions in semigroups related to free convolution are
infinitely divisible (see \cite{Anshelevich2012, Nica2009}).
\begin{prop}\label{subordINF}
 Let $\mu,\nu\in\Mreal$.
 \begin{enumerate}[$(1)$]
   \item We have that $\p{\mu\small{\subord}\nu}(z)=(\p{\mu}\circ\f{\nu})(z)$.
   \item If $\mu\in\mathcal{ID}(\boxplus,\mathbb{R})$, then $\mu\small{\subord}\nu\in\mathcal{ID}(\boxplus, \mathbb{R})$. In particular,
       $\gamma_t\small{\subord}\nu\in \mathcal{ID}(\boxplus, \mathbb{R})$ and $\p{\gamma_t\small{\subord}\nu}(z)=t\g{\nu}(z)$,
       where $\gamma_t$ is the semi-circular distribution with variance $t$.
    \item If $\nu=\mu\boxplus\nu'$ for $\nu'\in\Mreal$, then $\mu\small{\subord}\nu\in\mathcal{ID}(\boxplus, \mathbb{R})$. In particular, $\mu\small{\subord}\mu\in \mathcal{ID}(\boxplus, \mathbb{R})$, and $\p{\mu\small{\subord}\mu}(z)=z-\f{\mu}(z)$.
 \end{enumerate}
\end{prop}
\begin{proof}
Part (1) is Lemma 1 in \cite{Anshelevich2012}. Note that $\p{\gamma_t}(z)=t/z$ and
$(\p{\mu}\circ\f{\mu})(z)=z-\f{\mu}(z)$, Part (2) and (3) follow from Part 1 and Lemma 2 in \cite{Anshelevich2012}, see also Corollary 2.3 in \cite{arithmetic}.
\end{proof}

The following result was inspired by a question in \cite{Anshelevich2012}.
I am grateful to Michael Anshelevich for sending me a updated version of the paper \cite{Anshelevich2012}.
 \begin{lemma}\label{lemma:2.6}
  Given $\tau, \rho\in\Mreal$, if $\tau\small{\subord} \rho \in \mathcal{ID}(\boxplus, \mathbb{R})$, then $\rho\boxplus\tau^{\boxplus t}$ is
  defined for all $t\geq 0$ in the sense that $\p{\rho}+t\p{\tau}$ is the Voiculescu transform of a positive
measure. Moreover, we have that $\f{\rho\boxplus(\tau^{\boxplus t})}=\f{\rho}\left(\f{(\tau\small{\subord}\rho)^{\boxplus t}}(z) \right)$.
 \end{lemma}
 \begin{proof}
    Let $\sigma=\tau\small{\subord} \rho$, and $\sigma_t=\sigma^{\boxplus t}$. By Proposition \ref{fmu}, there exists a unique probability measure $\mu_t\in\Mreal$,
such that
    \begin{equation}\nonumber
      \f{\mu_t}=\f{\rho}\left(\f{\sigma_t}(z) \right).
    \end{equation}
   We claim that $\p{\mu_t}(z)=\p{\rho}(z)+t\p{\tau}(z)$. Indeed, by Proposition \ref{subordINF}, we have that
     \begin{equation}\nonumber
       \f{\sigma_t}^{-1}(z)-z=t\cdot \p{\sigma}(z)=t\cdot \p{\tau}\left( \f{\rho}(z)\right),
     \end{equation}
   and we thus obtain that
    \begin{align*}
      \p{\mu_t}\left(\f{\rho}(z) \right)&=\f{\mu_t}^{-1}\left(\f{\rho}(z) \right) - \f{\rho}(z)\\
                   &=\f{\sigma_t}^{-1}(z)-\f{\rho}(z)\\
                   &=\f{\sigma_t}^{-1}(z)-z+z-\f{\rho}(z)\\
                   &=t\cdot \p{\tau}\left(\f{\rho}(z) \right)+\f{\rho}^{-1}\left(\f{\rho}(z) \right)
                       -\f{\rho}(z).
    \end{align*}
    By analytic continuation, we conclude that
      \begin{equation}\nonumber
        \begin{split}
                 \p{\mu_t}(z)&=t\cdot \p{\tau}(z)+\f{\rho}^{-1}(z)-z\\
            &=\p{\rho}(z)+t\cdot \p{\tau}(z),
        \end{split}
      \end{equation}
which completes the proof.
      \end{proof}

\begin{rmk}
\emph{
There are examples $\rho, \tau\in\Mreal$ such that $\tau\small{\subord}\rho\in\mathcal{ID}(\boxplus, \mathbb{R})$
but $\tau \notin \mathcal{ID}(\boxplus, \mathbb{R})$ and $\tau$ is not a summand of $\rho$, see \cite{Anshelevich2012}.
}
\end{rmk}

Combining Proposition \ref{subordINF} and Lemma \ref{lemma:2.6}, we can reconstruct 
Nica-Speicher free convolution semigroups \cite{BB2004,NicaS1996} as follows.
\begin{thm}\label{thm:2.7}
 Given $\mu\in\Mreal$, then $\mu^{\boxplus t}\in\Mreal$ is
 defined such that $\p{\mu^{\boxplus t}}(z)=t\p{\mu}(z)$ for all $t>1$. Moreover,
 there exists an analytic map $\omega_t :\mathbb{C}^+\rightarrow \mathbb{C}^+$
such that $\f{\mu^{\boxplus t}}(z)=\f{\mu}(\omega_t(z))$, for $z\in\mathbb{C}^+$, $\omega_t=\f{(\mu\small{\subord}\mu)^{\boxplus (t-1)}}(z)$ and $\p{(\mu\small{\subord}\mu)^{\boxplus (t-1)}}=(t-1)(z-\f{\mu}(z))$ for all $t>1$.
\end{thm}

Let $\mathit{H}_t(z)=z+(t-1)(z-\f{\mu}(z))$, then by Proposition \ref{leftInvreal} and Theorem \ref{thm:2.7}, we know that $\mathit{H}_t$ is the left inverse
of $\omega_t$ such that
$\mathit{H}_t(\omega_t(z))=z$ for $z\in\mathbb{C}^+$. Therefore, for $t>1$, $\omega_t(z)$ can be written as
\begin{equation}\label{omegaT}
  \omega_t(z)=z+\left(1-\frac{1}{t}\right)(\f{\mu^{\boxplus t}}(z)-z),\,\,\,z\in\mathbb{C}^+.
\end{equation}
We deduce from (\ref{omegaT}) and the definition of $\omega_t$ in Theorem \ref{thm:2.7} that, for $t>0$,
 \begin{equation}\nonumber
   z-\f{(\mu\small{\subord}\mu)^{\boxplus t}}(z)=\left(1-\frac{1}{t+1}\right)(z-\f{\mu^{\boxplus (t+1)}}(z)),
 \end{equation}
which implies that
 \begin{equation}\label{subordB}
 (\mu\small{\subord}\mu)^{\boxplus t}(z)=(\mu^{\boxplus (t+1)})^{\uplus(t/(t+1))}.
 \end{equation}
\subsection{Two formulas related to free Brownian motion}
Given $\mu\in\Mreal$, we construct subordination functions
$\omega_t$
as in Theorem \ref{thm:2.7}. Let $\sigma_t=(\mu\small{\subord}\mu)^{\boxplus t}\in\Mreal$, then $\omega_{t+1}=\f{\sigma_t}(z)$ for $t>0$. Given $\nu\in\Mreal$, let $\rho_t=\gamma_t\subord\nu$ and let
$\mathit{F}_t=\f{\rho_t}(z)$ for all $t>0$.
From Proposition \ref{subordINF} and Theorem \ref{thm:2.7}, we know that
$\rho_t$ and $\sigma_t$ are $\boxplus$-infinitely divisible and their Voiculescu transforms
are given by $\p{\rho_t}(z)=t\g{\nu}$ and $\p{\sigma_t}(z)=t(z-\f{\mu}(z))$.
By comparing
Voiculescu transform of $\rho_t$ with Voiculescu transform
of $\sigma_t$, we deduce that
$\mathit{F}_t=\omega_{t+1}$ for some $t>0$ if and only if
$\g{\nu}(z)=z-\f{\mu}(z)$.

For any $t>0$, Belinschi and Nica \cite{BN2008Ind} construct the transformation $\mathbb{B}_t:\Mreal\rightarrow \Mreal$ such that
\begin{equation}\nonumber
  \mathbb{B}_t({\mu})=(\mu^{\boxplus(1+t)})^{1/1+t}, \,\text{for}\, \mu\in\Mreal.
\end{equation}
They also show that $\mathbb{B}_t$ is a semigroup and $\mathbb{B}_2=\mathbb{B}$,
where the map $\mathbb{B}:\mathcal{ID}(\uplus, \mathbb{R}) \rightarrow \mathcal{ID}(\boxplus, \mathbb{R})$ is the bijective map from the $\uplus$-infinitely divisible distributions
to the $\boxplus$-infinitely divisible distributions, discovered in the seminal paper \cite{BPata1999}. The following theorem is from \cite{BN2008Ind}.
\begin{thm}\label{thm:2.8}
 Let $\mu$ and $\nu$ be a pair of probability measures on the real line such that
 \begin{equation}\label{relationPlus}
   \g{\nu}(z)=z-\f{\mu}(z),\,\,\,z\in\mathbb{C}^+.
 \end{equation}
 Then we have
 \begin{equation}\nonumber
   \g{\nu\boxplus\gamma_t}(z)=z-\f{\mathbb{B}_t(\mu)}(z), \,\,\, t>0, \,\,\,z\in\mathbb{C}^+.
 \end{equation}
\end{thm}

\begin{rmk}
\emph{
 Given $\mu, \nu\in\Mreal$ satisfying
 (\ref{relationPlus}), then Maassen \cite{Maassen} shows that $\mu$ has mean zero and variance one. Conversely, if $\mu\in\Mreal$ has mean zero and variance one, then there exists a unique $\nu\in\Mreal$ satisfying (\ref{relationPlus}).
}
\end{rmk}

Given $\tau\in\mathcal{ID}(\boxplus, \mathbb{R}) $ and $\mu, \nu\in\Mreal$, we compare
free L\'{e}vy process $\nu\boxplus\tau^{\boxplus t}$ and free convolution semigroup $\mu^{\boxplus (t+1)}$.
If $\p{\tau}\left(\f{\nu}(z) \right)=z-\f{\mu}(z)$, then $\tau\subord\nu=\mu\subord\mu$, which implies that subordination function of $\f{\nu\boxplus(\tau^{\boxplus t})}$ to $\f{\nu}$ is the same as the subordination function of $\f{\mu^{\boxplus (t+1)}}$ to $\f{\mu}$.
The following theorem generalizes Theorem \ref{thm:2.8}. The argument is similar to the proof
of Theorem 1.6 in \cite{BN2008Ind} (see also the proof of Lemma 3 in \cite{Anshelevich2012}), therefore we omit the proof.
\begin{thm}\label{levyBN}
 Given $\tau\in\mathcal{ID}(\boxplus, \mathbb{R})$, and let $\mu$ and $\nu$ be a pair of probability measures
 on the real line such that
      \begin{equation}\nonumber
         \p{\tau}\left(\f{\nu}(z) \right)=z-\f{\mu}(z),\, z\in \mathbb{C}^+.
      \end{equation}
 Then we have
      \begin{equation}\nonumber
         \p{\tau}\left(\f{\nu\boxplus (\tau^{\boxplus t})}(z) \right)
              =z-\f{\mathbb{B}_t(\mu)}(z), \,t>0, \,z\in \mathbb{C}^+.
      \end{equation}
\end{thm}
\begin{rmk}
\emph{
Let $\tau=\gamma_{a, b}$ be the semi-circular distribution with mean $a$ and variance $b$,
and let $\mu, \nu$ be a pair of probability measures on the real line such that
      \begin{equation}
        \p{\tau}\left(\f{\nu}(z) \right)=z-\f{\mu}(z).
      \end{equation}
We first compute
       \begin{equation}\label{eq:2a}
         \begin{split}
          \f{\mu}(z) & =z- \p{\tau}\left(\f{\nu}(z) \right)\\
                     & = z- \left(a+\frac{b}{z}\right)\circ \f{\nu}(z)\\
                      & = z- a -b\g{\nu}(z).
         \end{split}
       \end{equation} }
\emph{By Theorem \ref{levyBN}, then we have that
       \begin{equation}\label{eq:2b}
          \begin{split}
          \f{\mathbb{B}_t(\mu)}(z) & = z-\p{\tau}\left(\f{\nu\boxplus\tau^{\boxplus t}}(z) \right) \\
                & = z - \left(\left(a+\frac{b}{z}\right)\circ \f{\nu\boxplus \gamma_{a, b}^{\boxplus t}} \right)(z) \\
                 & = z-a-b\g{\nu\boxplus \gamma_{a, b}^{\boxplus t}}(z).
       \end{split}
       \end{equation} }
\emph{By (\ref{eq:2b}) and the definition of Boolean convolution, we obtain that
       \begin{equation}\label{eq:2c}
          \f{(\mathbb{B}_t(\mu))^{\uplus t}}(z)=z-ta-tb\g{\nu\boxplus\gamma_{a, b}^{\boxplus t}}(z).
       \end{equation}
      Equations (\ref{eq:2a}), (\ref{eq:2b}) and (\ref{eq:2c}) were studied in \cite{Anshelevich2012}.
      We would like to point out that, as it was shown in \cite{Anshelevich2012} (see Proposition 1 and Example 1), $(\mathbb{B}_t(\mu))^{\uplus t}\in\mathcal{ID}(\boxplus, \mathbb{R})$, and
      $(\mathbb{B}_t(\mu))^{\uplus t}=(\tau\subord\nu)^{\boxplus t}=(\mu\subord\mu)^{\boxplus t}$.
      In fact, for all $\mu\in\Mreal$, we can deduce from (\ref{subordB}) and the identity $(\mathbb{B}_t(\mu))^{\uplus t}=\left(\mu^{\boxplus (1+t)}\right)^{\uplus t/(1+t)}$ that $(\mathbb{B}_t(\mu))^{\uplus t}$ is the measure associated with the subordination function of $\mu^{\boxplus (1+t)}$ with respect to $\mu$, that is
$(\mathbb{B}_t(\mu))^{\uplus t}=(\mu\subord\mu)^{\boxplus t}$. }
\end{rmk}
\section{Multiplicative free convolution and Multiplicative Boolean convolution on $\MT$}
Given any two probability measures $\mu, \nu$ on $\mathrm{T}$, the unit circle of $\mathbb{C}$,
we can define
their multiplicative free convolution.
We first recall the calculation of the multiplicative free convolution
of two measures on $\mathrm{T}$ with nonzero means.
Given $\mu\in\MT$, we define
\begin{equation}\nonumber
  \pp{\mu}(z)=\int_{\mathrm{T}}\frac{tz}{1-tz}d\,\mu(t)
\end{equation}
and set $\et{\mu}(z)=\pp{\mu}(z)/(1+\pp{\mu}(z))$. The following proposition \cite{BB2005} characterizes the
$\eta$-transforms of probability measures on $\mathrm{T}$.
\begin{prop}\label{etaD}
  Let $\eta : \mathbb{D}\rightarrow \mathbb{C}$ be an analytic function. Then the following
assertions are equivalent.
 \begin{enumerate}[$(1)$]
   \item There exists a probability measure $\mu\in\MT$ such that $\eta=\et{\mu}$.
   \item $\eta(0)=0$, and $|\eta(z)|<1$ holds for all $z\in \mathbb{D}$.
 \end{enumerate}
\end{prop}

If $\mu\in \MT\cap \Mstar$, then $\et{\mu}'(0)=\int_{\mathrm{T}}t d\,\mu(t) \neq 0$.
Therefore, the inverse $\eta_{\mu}^{-1}$ is defined in a neighborhood of zero.
We set $\Sigma_{\mu}(z)=\eta_{\mu}^{-1}(z)/z$. Given $\mu , \nu \in \MT\cap\Mstar$,
their multiplicative free convolution, which is denoted by $\mu\boxtimes \nu$,
is the unique probability measure in $\MT\cap\Mstar$ such that
\begin{equation}\label{sigmaM}
  \Sigma_{\mu\boxtimes\nu}(z)=\Sigma_{\mu}(z)\Sigma_{\nu}(z)
\end{equation}
holds for $z$ in a neighborhood of zero.

It is known \cite{Biane1998, BB2007new} that there exist two analytic
functions $\omega_1, \omega_2 : \mathbb{D}\rightarrow \mathbb{D}$ such that
\begin{enumerate}[$(1)$]
  \item $\omega_1(0)=\omega_2(0)=0$,
  \item $\eta_{\mu\boxtimes\nu}(z)=\eta_{\mu}(\omega_1(z))=\eta_{\nu}(\omega_2(z))$.
\end{enumerate}

A probability measure $\mu\in\MT$ is said
to be $\boxtimes$-infinitely divisible if for any positive integer $n$,
there exists $\mu_n\in\MT$ such that
$\mu=(\mu_n)^{\boxtimes n}=\mu_n \boxtimes\cdots\boxtimes\mu_n$.
It is shown in \cite{BV1992} that
if $\mu\in\MT\backslash\Mstar$ is $\boxtimes$-infinitely divisible,
then $\mu$ is the Haar measure on $\mathrm{T}$;
and $\mu\in\MT\cap\Mstar$ is $\boxtimes$-infinitely divisible
if and only if there exists
a function
\begin{equation}\label{levyMultiT}
  u(z)=\alpha i +\int_{\mathrm{T}}\frac{1+tz}{1-tz}d\,\sigma(t),
\end{equation}
such that $\Sigma_{\mu}(z)=\exp(u(z)) $,
where $\alpha\in\mathbb{R}$ and $\sigma$ is a finite positive measure on $\mathrm{T}$.
Equation (\ref{levyMultiT}) is the analogue of the L\'{e}vy-Hin\v{c}in
formula for multiplicative free convolution on $\mathrm{T}$.
The analogue of the normal distribution
in this context is given by $\Sigma_{\lambda_t}(z)=\exp\left(\frac{t}{2}\frac{1+z}{1-z}\right)$.
Denote by $\mathcal{ID}(\boxtimes, \mathrm{T})$ the set of all $\boxtimes$-infinitely divisible
measures on $\mathrm{T}$.

\begin{lemma}\label{infMstar}
 If $\mu\in\MT\cap\Mstar$ is $\boxtimes$-infinitely divisible.
 Then
 \begin{enumerate}[$(1)$]
   \item The function $H(z)=z\Sigma_{\mu}(z)$ is the left inverse of $\et{\mu}(z)$, that is
        $H(\et{\mu}(z))=z$ for all $z\in\mathbb{D}$.
   \item The function $\et{\mu}$ extends to be a continuous function on $\overline{\mathbb{D}}$, and $\et{\mu}$ is one-to-one on $\overline{\mathbb{D}}$.
   \item The set
          $\{z\in \mathbb{D}: |z \Sigma_{\mu}(z)|<1\}$ is a simply connected domain which coincides with $\{\eta_{\mu}(z): z\in \mathbb{D} \}$, and its boundary is $\eta_{\mu}(\mathit{T})$ which is a simple closed curve.
 \end{enumerate}
\end{lemma}
\begin{proof}
Observing that $H(\eta_{\mu}(z))=z$
is valid in a neighborhood of zero, we obtain assertion (1) by analytic continuation.

Note that $H:\mathbb{D}\rightarrow \mathbb{C}$ satisfies
the conditions in Proposition 4.5 in \cite{BB2005} and thus assertions (2) and (3) hold.
\end{proof}
\subsection{Multiplicative free Brownian motion}
For $\mu\in \MT$ and $t>0$, we study the multiplicative free convolution
$\mu\boxtimes \lambda_t$.
We first concentrate on the case when $\mu$ has
nonzero mean. The case when $\mu$ has mean zero will be studied in Subsection 3.2.

We start with the following result which is the multiplicative version of Lemma 1 in \cite{Biane1997}.
\begin{lemma} \label{subLemma}
 Given $\mu,\nu \in \MT\cap\Mstar$, we have that
 \begin{equation}\nonumber
   \eta_{\mu}(z)=\eta_{\mu\boxtimes\nu}(z\cdot \Sigma_{\nu}(\eta_{\mu}(z)))
 \end{equation}
holds for $z$ in a neighborhood of zero.
\end{lemma}
\begin{proof}
From (\ref{sigmaM}), we find that
 \begin{equation}\label{sigmaInv}
   \frac{\eta_{\mu\boxtimes\nu}^{-1}(z)}{z}=\frac{\eta_{\mu}^{-1}(z)}{z}\cdot
       \frac{\eta_{\nu}^{-1}(z)}{z}
 \end{equation}
holds for $z$ in a neighborhood of zero, which we denote by $\mathit{D}_0$. We choose
 a subdomain $\mathit{D}_1\subset \mathit{D}_0$ such that
 $\eta_{\mu}(\mathit{D}_1)\subset \mathit{D}_0$. Replacing $z$
 by $\eta_{\mu}(z)$ in (\ref{sigmaInv}), we obtain that
    \begin{equation}\label{eq:3.3a}
      \frac{\eta_{\mu\boxtimes\nu}^{-1}(\eta_{\mu}(z))}{\eta_{\mu}(z)}=
       \frac{\eta_{\mu}^{-1}(\eta_{\mu}(z))}{\eta_{\mu}(z)}\cdot
       \frac{\eta_{\nu}^{-1}(\eta_{\mu}(z))}{\eta_{\mu}(z)}
       =\frac{z}{\eta_{\mu}(z)}\cdot\frac{\eta_{\nu}^{-1}(\eta_{\mu}(z))}{\eta_{\mu}(z)}
    \end{equation}
holds for $z\in \mathit{D}_1$. Note that $\eta_{\nu}^{-1}(z)=z\Sigma_{\nu}(z)$ holds for $z\in\mathit{D}_0$,
and we then rewrite (\ref{eq:3.3a}) as
 \begin{equation}\label{eq:3.3b}
   \eta_{\mu\boxtimes\nu}^{-1}(\eta_{\mu}(z))=z\Sigma_{\nu}(\eta_{\mu}(z)).
 \end{equation}
Applying $\eta_{\mu\boxtimes\nu}$ on both sides of (\ref{eq:3.3b}) yields that
 \begin{equation}\nonumber
   \eta_{\mu}(z)=\eta_{\mu\boxtimes\nu}(z\Sigma_{\nu}(\eta_{\mu}(z)))
 \end{equation}
holds for $z$ in a neighborhood of zero $\mathit{D}_1$.
\end{proof}

For any $t>0$, we denote by $\eta_t:\mathbb{D}\rightarrow\mathbb{D}$ the
subordination function of $\mu\boxtimes\lambda_t$ with respect to $\mu$.
Since $\eta_t :\mathbb{D}\rightarrow \mathbb{D}$ is analytic and $\eta_t(0)=0$,
Proposition \ref{etaD} implies the existence of a probability measure $\rho_t$ such that
$\eta_{\rho_t}(z)=\eta_t(z)$.
\begin{lemma}\label{rhoInf}
The measure $\rho_t$ is $\boxtimes$-infinitely divisible and its $\Sigma$-transform is
$\Sigma_{\rho_t}(z)=\Sigma_{\lambda_t}(\eta_{\mu}(z))$.
\end{lemma}
\begin{proof}
 Define analytic function $\Phi_t:\mathbb{D}\rightarrow\mathbb{C}$ by
  $\Phi_t(z):=z\Sigma_{\lambda_t}(\eta_{\mu}(z)) $ for all $t>0$.
  By Lemma \ref{subLemma}, we have that
    \begin{equation}\nonumber
       \et{\mu}(z)=\et{\mu\boxtimes\lambda_t}(z\Sigma_{\lambda_t}(\eta_{\mu}(z)))=\et{\mu\boxtimes\lambda_t}(\Phi_t(z))
    \end{equation}
which implies that
 \begin{equation}\nonumber
   \eta_{\mu\boxtimes\lambda_t}(z)=\eta_{\mu}(\eta_t(z))
       =\eta_{\mu\boxtimes\lambda_t}(\Phi_t(\eta_t(z))).
 \end{equation}
Since $\eta_{\mu\boxtimes\lambda_t}$ is invertible in a neighborhood
of zero, we have that $\Phi_t(\eta_t(z))=z$ in a neighborhood
of zero.

We thus obtain that
  $\eta_{\rho_t}^{-1}(z)=\eta_t^{-1}(z)=\Phi_t(z)$ holds for $z$ in a neighborhood of zero, which yields that
   \begin{equation}\label{eq:3.5}
     \Sigma_{\rho_t}(z)=\frac{\eta_{\rho_t}^{-1}(z)}{z}=\Sigma_{\lambda_t}\left(\eta_{\mu}(z) \right).
   \end{equation}
By the definition of the $\psi$- and $ \eta$-transforms, we have that
 \begin{equation}\label{formula}
   \Sigma_{\lambda_t}(\eta_{\mu}(z))=
   \exp\left( \frac{t}{2} \int_{\mathrm{T}}\frac{1+\xi z}{1-\xi z} d\mu(\xi)\right).
 \end{equation}
The real part of the integrand in (\ref{formula}) is positive for all $z\in\mathbb{D}$, then
the asseration follows from (\ref{eq:3.5}) and Theorem 6.7 in \cite{BV1992}.
\end{proof}

By (\ref{eq:3.5}), the right hand side of (\ref{formula}) is the
L\'{e}vy-Hin\v{c}in representation of $\rho_t$. We can also write $\eta_t$ in terms of $\lambda_t$ and $\mu\boxtimes\lambda_t$.
Replacing $z$ by $\eta_{\mu\boxtimes\lambda_t}(z)$ in the equation
\begin{equation}\nonumber
    \frac{\eta_{\mu\boxtimes\lambda_t}^{-1}(z)}{z}=\frac{\eta_{\mu}^{-1}(z)}{z}\cdot
       \frac{\eta_{\lambda_t}^{-1}(z)}{z},
 \end{equation}
we obtain that
 \begin{equation}\nonumber
   \frac{z}{\eta_{\mu\boxtimes\lambda_t}(z)}=\frac{\eta_t(z)}
   {\eta_{\mu\boxtimes\lambda_t}(z)}\cdot \Sigma_{\lambda_t}
   \left(\eta_{\mu\boxtimes\lambda_t}(z)\right),
 \end{equation}
which shows that
 \begin{equation}
   \eta_t(z)=\frac{z}{\Sigma_{\lambda_t}
   \left(\eta_{\mu\boxtimes\lambda_t}(z)\right)}.
 \end{equation}
\subsection{Modified $\mathcal{S}$-transform and subordination functions}
 Given $\mu\in\MT\backslash\Mstar$ and $\nu\in\MT\cap\Mstar$, it is known from \cite{Biane1998} that $\et{\mu\boxtimes\nu}$ is subordinated to $\et{\mu}$ and $\et{\nu}$. The subordination function for this case is generally not unique (see example \ref{example:3.5} below). However, we show that there is a nice subordination function, which we call the principal subordination function, uniquely determined by certain conditions.
Using the principal subordination function, results related to subordination function in the case $\mu, \nu\in\MT\cap\Mstar$ can be extended to the case where $\mu\in\MT\backslash\Mstar$ and $\nu\in\MT\cap\Mstar$.

Let us first give an example which illustrates the non-uniqueness of subordination functions.
 \begin{example}\label{example:3.5}
  For $k\in\mathbb{N}$, and let $\lambda^{(k)}=1/k \sum_{n=0}^{k-1}\delta_{z_n}$, where $z_n=e^{2\pi i n/k}$.
  We have $\psi_{\lambda^{(k)}}(z)=z^k/(1-z^k)$ and $\et{\lambda^{(k)}}(z)=z^k$. Given $\nu\in\MT\cap\Mstar$,
  if $\omega:\mathbb{D}\rightarrow\mathbb{D}$ is a subordination function of $\et{\lambda^{(k)}\boxtimes \nu}$ with respect to $\et{\lambda^{(k)}}$, then $\omega^{(n)}(z):=e^{2\pi i n/k}\omega(z)$ is also a subordination function of $\et{\lambda^{(k)}\boxtimes \nu}$ with respect to $\et{\lambda^{(k)}}$ for all integer $0<n<k$.
 \end{example}

 We now introduce the modified $\mathcal{S}$-transform. Given two free random variables $x$ and $y$ in a W*-probability space $(\mathcal{A}, \phi)$, such that $\phi(x)=0$ and $\phi(y)\neq 0$, we can not directly apply Voiculescu's $\mathcal{S}$-transform ($\Sigma$-transform) to calculate the distribution of $xy$.
  N. Raj Rao and R. Speicher \cite{RS2007} introduce a new transform, which we call the modified $\mathcal{S}$-transform, to deal with this case. They apply the modified $\mathcal{S}$-transform to study the distribution of $xy$ where $x, y$ are free self-adjoint random variables such that $\phi(x)=0$, $\phi(y)\neq 0$. For nonzero self-adjoint operator $x$, we have that $\phi(x^2)\neq 0$.
Assume that $\phi(x)=\cdots=\phi(x^{k-1})=0$ and $\phi(x^k)\neq 0$, Arizmendi \cite{Ariz2012} observe that we can calculate the distribution of $xy$ using the idea in \cite{RS2007}.
  We present the details of their work for reader's convenience.

  We first recall some definitions.
  For $\mu\in\MT\cap\Mstar$, we have $\psi_{\mu}(0)=0$ and $\psi_{\mu}'(0)\neq 0$. It follows that there
  exists a function $\chi_{\mu}(z)$, which is analytic in a neighborhood of zero, such that
     \begin{equation}\nonumber
        \psi_{\mu}(\chi_{\mu}(z))=\chi_{\mu}(\psi_{\mu}(z))=z
     \end{equation}
holds for sufficiently small $z$. The usual $\mathcal{S}$-transform is defined by
    \begin{equation}\nonumber
     \mathcal{S}_{\mu}(z)=\frac{z+1}{z}\chi_{\mu}(z).
    \end{equation}
   We then have
     \begin{equation} \nonumber
     \Sigma_{\mu}(z)=\mathcal{S}_{\mu}\left(\frac{z}{1-z} \right), \,\eta_{\mu}^{-1}(z)=\chi\left(\frac{z}{1-z}\right).
      \end{equation}
    We set
     \begin{equation}\nonumber
    \MT^k=\left\{\mu\in\MT: \int_{\mathrm{T}}t^n d\mu(t)=0\, \text{for} \,1\leq n < k,\,
\text{and} \int_{\mathrm{T}}t^k d\mu(t)\neq 0  \right\}.
     \end{equation}
Then for $\mu\in\MT^k$, we have that
  \begin{equation}
    \begin{cases}
       \psi_{\mu}'(0)=\cdots\psi_{\mu}^{(k-1)}(0)=0=\eta_{\mu}'(0)=\cdots\eta_{\mu}^{(k-1)}(0),\\
       \psi_{\mu}^{(k)}(0)\neq 0 \,\,\text{and}\,\,\eta_{\mu}^{(k)}\neq 0.
    \end{cases}
  \end{equation}
For $\mu\in\MT^k, \nu\in\MT\cap\Mstar$, from the definition of free independence, we deduce that $\mu\boxtimes\nu\in\MT^k$.

We recall the following classical result in complex analysis
(see, for example, \cite{Hille}).
\begin{thm}\label{inverseF}
 If $f(z)$ is holomorphic in $|z|<R$, and suppose that
 \begin{equation}\nonumber
   f(0)=f'(0)=\cdots=f^{(k-1)}(0)=0, \,\,\,\, f^{(k)}\neq 0,
 \end{equation}
 then for small values of $w\neq 0$ the equation
 \begin{equation}\nonumber
   f(z)=w
 \end{equation}
 has $k$ roots $z_1(w)$, $\cdots$, $z_k(w)$, which tend to zero when $w$ tends to zero.
 Moreover, there exists a function $g(w)$, holomorphic for $w$ sufficiently
 small with $g(0)=0$ and $g'(0)\neq 0$, such that for any fixed small values $w\neq 0$,
  \begin{equation}\nonumber
    z_j(w)=g(\omega^j w^{1/k}), \,\,\, \omega=e^{2\pi i/k},\,\,\,0\leq \arg{\omega^{1/k}}<\frac{2\pi}{k},
   \end{equation}
 if we put those roots in a certain order.
\end{thm}

\begin{rmk}\label{rmk:4}
\emph{The converse of Theorem \ref{inverseF} is also true. More precisely, if we
are given a function $g(w)$ which is holomorphic for $w$ sufficiently small with $g(0)=0$ and $g'(0)\neq 0$, and for $j=1, \cdots, k$, let
  \begin{equation}\nonumber
    z_j(w)=g(\omega^j w^{1/k}), \,\,\, \omega=e^{2\pi i/k},\,\,\,0\leq \arg{\omega^{1/k}}<\frac{2\pi}{k},
   \end{equation}
then $z_1(w), \cdots, z_k(w)$ are the roots of the equation
 \begin{equation}\nonumber
   F^k(z)=w,
 \end{equation}
where $F$ is a holomorphic function defined in a neighborhood of the zero such that $F(g(w))=w$. }
\end{rmk}

For $j=1,\cdots, k$, denote $D_{j,r}=\{\omega^j z: 0\leq \arg(z)<2\pi/k, |z|<r \}$.
We record the following result for convenience.
\begin{prop}\label{identityINV}
  Under the assumption of Theorem \ref{inverseF}, we have that $z_j\left(f(z)\right)=z$ for
  $z\in g(D_{j, r})$ for $r$ sufficiently small.
\end{prop}

Given $\mu\in\MT^k$, and by Theorem \ref{inverseF}, we know that
there exist $k$ functions represented by the power series in $z^{1/k}$ such that
\begin{equation}\label{eq:chiINV}
  \psi_{\mu}(\chi_{\mu}^{(j)}(z))=z,
\end{equation}
for $z$ sufficiently small. Moreover, there exists a function $g_{\mu}(w)$ holomorphic in
a neighborhood of the zero, such that for $j=1,\cdots, k$,
\begin{equation}\nonumber
  \chi_{\mu}^{(j)}(z)=g_{\mu}(\omega^{j}z^{1/k}),
\end{equation}
where $\omega=e^{2\pi i/k},\,\,\,0\leq \arg{z^{1/k}}<2\pi/k$.

\begin{defn}\label{modified}
Given $\mu\in \MT^k$. Let $\chi_{\mu}^{(j)}$ be the inverse function of $\psi_{\mu}$
in (\ref{eq:chiINV}), the modified $\mathcal{S}$-transform of $\mu$ is $k$ functions
$\mathcal{S}_{\mu}^{(1)}(z), \cdots, \mathcal{S}_{\mu}^{(k)}(z)$, such that
for $j=1,\cdots,k$,
\begin{equation}\nonumber
  \mathcal{S}_{\mu}^{(j)}(z)=\chi_{\mu}^{(j)}(z) \cdot \frac{1+z}{z}.
\end{equation}
\end{defn}

 Given $\mu\in\MT^k$ and $\nu\in\MT\cap\Mstar$, we set
\begin{equation}\nonumber
  \mathcal{S}^{(j)}(z)=\mathcal{S}_{\mu}^{(j)}(z)\cdot \mathcal{S}_{\nu}(z)
\end{equation}
and compute
\begin{equation}\label{eq:3.10}
   \begin{split}
  \chi^{(j)}(z)&=\mathcal{S}^{(j)}(z)\cdot\frac{z}{1+z}\\
            &=\mathcal{S}_{\mu}^{(j)}(z)\cdot \mathcal{S}_{\nu}(z)\cdot\frac{z}{1+z}\\
            &=\chi_{\mu}^{(j)}(z)\cdot \mathcal{S}_{\nu}(z)\\
            &=g_{\mu}(\omega^{j}z^{1/k})\cdot \mathcal{S}_{\nu}(z)\\
            &=g(\omega^{j}z^{1/k}),
\end{split}
\end{equation}
where $g(z)=g_{\mu}(z)\cdot \mathcal{S}_{\nu}(z^k)$ is a function such that $g(0)=0, \,g'(0)\neq 0$.
From Remark \ref{rmk:4}, we deduce that for different $j$, there
exists the same left inverse $\psi$ such that $\psi(\chi^{(j)}(z))=z$.
Therefore, we have the following proposition.
\begin{prop}\label{prop:3.10}
Given $\mu \in \MT^k$ and $\nu \in \MT \cap \Mstar$,
for $1\leq j\leq k$, let
\begin{align}
  \mathcal{S}^{(j)}(z)&=\mathcal{S}_{\mu}^{(j)}(z)\cdot \mathcal{S}_{\nu}(z)\nonumber\\
  \chi^{(j)}(z)&=\mathcal{S}^{(j)}(z)\cdot\frac{z}{1+z}.\nonumber
\end{align}
Then there exists a unique holomorphic function $\psi$ defined in a neighborhood
of the zero such that
\begin{equation}\nonumber
  \psi\bigl((\chi^{(j)})(z)\bigr)=z.
\end{equation}
\end{prop}

The following result is due to Raj Rao and Speicher \cite{RS2007} and Arizmendi \cite{Ariz2012}.
\begin{thm}\label{thm:3.11}
Given $\mu\in \MT^k, \nu\in\MT\cap\Mstar$,
we have that
\begin{equation}\label{equation}
 \mathcal{S}_{\mu\boxtimes\nu}^{(j)}(z)=\mathcal{S}_{\mu}^{(j)}(z)\cdot\mathcal{S}_{\nu}(z), \,\,
 j=1,\cdots,k,
\end{equation}
where the modified $\mathcal{S}$-transforms are listed in a certain order.
\end{thm}

Because of Proposition \ref{prop:3.10} and Theorem \ref{thm:3.11}, for fixed $\mu\in\MT^k$ and $\nu\in\MT\cap\Mstar$, we denote
\begin{equation}\label{eq:3.12}
   \psi(z)=\pp{\mu\boxtimes\nu}(z), \text{and}\, \chi^{(j)}(z)=\chi^{(j)}_{\mu\boxtimes\nu}(z),
\end{equation}
and we also denote $g(z)=g_{\mu}(z)\cdot \mathcal{S}_{\nu}(z^k)$ as in (\ref{eq:3.10}).

Given $\mu\in\MT^k, \nu\in\MT\cap\Mstar$, we set $\iota_{\mu}^{(j)}(z)=\chi_{\mu}^{(j)}(z/(1-z)) $ and
  $\iota_{\nu}(z)=\chi_{\nu}(z/(1-z))$. Theorem \ref{thm:3.11} implies that
  \begin{equation}\label{plug}
     \chi_{\mu\boxtimes\nu}^{(j)}(z)=\chi_{\mu}^{(j)}(z)\cdot\chi_{\nu}(z)\cdot\frac{1+z}{z}.
  \end{equation}
   We also have that $\chi_{\mu}^{(j)}(z)=g_{\mu}(\omega^j z^{1/k})$ and $\chi_{\mu\boxtimes\nu}^{(j)}(z)=g_{\mu}(\omega^j z^{1/k})\cdot \mathcal{S}_{\nu}(z)=g(\omega^j z^{1/k})$.
  Substituting $z$ by $\pp{\mu\boxtimes\nu}(z)$ in (\ref{plug}), and applying Proposition \ref{identityINV},
  we find that
    \begin{equation}\nonumber
      z=\chi_{\mu\boxtimes\nu}^{(j)}(\pp{\mu\boxtimes\nu}(z))=\chi_{\mu}^{(j)}(\pp{\mu\boxtimes\nu}(z))\cdot\chi_{\nu}(\pp{\mu\boxtimes\nu}(z))\cdot\frac{1+\pp{\mu\boxtimes\nu}(z)}{\pp{\mu\boxtimes\nu}(z)},
    \end{equation}
  where $z\in g(D_{j, r})$ for $r$ sufficiently small. We thus have that
    \begin{equation}\label{equFORsub}
      z \et{\mu\boxtimes\nu} =  \iota_{\mu}^{(j)}(\et{\mu\boxtimes\nu})\cdot \iota_{\nu}(\et{\mu\boxtimes\nu})
    \end{equation}
  holds in the same domain.

  We can now utilize the argument in \cite{BB2007new} to prove the existence of subordination function of $\et{\mu\boxtimes\nu}$ with respect to $\et{\mu}$ for $\mu\in\MT^k, \nu\in\MT\cap\Mstar$.
 Note that part of the following result is known in \cite{Biane1998}.
 \begin{thm}\label{BBnew}
  Given $\mu\in\MT^k, \nu\in\MT\cap\Mstar$, there exists two unique analytic functions $\omega_1, \omega_2: \mathbb{D}\rightarrow \mathbb{D}$ such that
   \begin{enumerate}[$(1)$]
     \item $\omega_1(0)=\omega_2(0)=0$;
     \item $\et{\mu\boxtimes\nu}(z)=\et{\mu}(\omega_1(z))=\et{\nu}(\omega_2(z))$,
     \item $\omega_1(z)\omega_2(z)=z\et{\mu\boxtimes\nu}(z)$ for all $z\in\mathbb{D}$.
   \end{enumerate}
 \end{thm}
 \begin{proof}
  Since $\et{\mu}(0)=0, \et{\nu}(0)=0$, we can write $\et{\mu}(z)=z f_1(z), \et{\nu}(z)=z f_2(z)$
  for two analytic functions $f_1, f_2: \mathbb{D}\rightarrow\mathbb{D}$. Fix $1\leq j\leq k$, set $\omega_1(z)=\iota_{\mu}^{(j)}(\et{\mu\boxtimes\nu}(z))$, $\omega_2(z)=\iota_{\nu}(\et{\mu\boxtimes\nu}(z))$ defined in $g(D_{j,r})$ for $r$ sufficiently small.

  By (\ref{equFORsub}), we have that
     \begin{equation}\nonumber
        z\et{\mu\boxtimes\nu}(z)=\omega_1(z)\omega_2(z),
     \end{equation}
 holds for $z\in g(D_{j,r})$. We thus obtain that
          \begin{equation}\nonumber
             \omega_1(z)=\frac{z\et{\mu\boxtimes\nu}}{\omega_2(z)}=\frac{z\et{\nu}(\omega_2(z))}{\omega_2(z)}=zf_2(\omega_2(z)).
          \end{equation}
   Similarly, we have that $\omega_2(z)=zf_1(\omega_1(z))$ for $z\in g(D_{j,r})$.
   Regarding $\omega_1(z), \omega_2(z)$ as Denjoy-Wolff points, the same argument in \cite{BB2007new} implies $\omega_1,\omega_2$ can be extended analytically to $\mathbb{D}$. By the uniqueness of Denjoy-Wolff points, $\omega_1, \omega_2$ does not depend on the choice of $j$.

   By the definitions of $\iota_{\mu}^{(j)}, \iota_{\nu}$, we have that $\et{\mu}(\omega_1(z))=\et{\nu}(\omega_2(z))=\et{\mu\boxtimes\nu}$ for $z$ in $g(D_{j,r})$. Thus (2) and (3) hold by analytic continuation. Since $\et{\nu}'(0)\neq 0$, $\et{\nu}$ is locally invertible near the origin and therefore $\omega_2$ is unique. Finally (3) implies the uniqueness of $\omega_1$.
 \end{proof}

Since $\mu \in\MT^k$ and  $\mu\boxtimes\nu \in \MT^k$, we have that $\omega_1'(0)\neq 0$, where $\omega_1$ is given in Theorem \ref{BBnew}.
 \begin{defn}
   For $\mu\in\MT^k, \nu\in\MT\cap\Mstar$, the subordination $\omega_1$ satisfying the relations (1), (2) and (3)
   in Theorem \ref{BBnew} is called the \emph{principal} subordination function of $\et{\mu\boxtimes\nu}$ with respect to $\et{\mu}$. The measure $\rho\in\MT\cap\Mstar$ satisfying $\et{\rho}(z)=\omega_1(z)$ is called the \emph{principal} subordination distribution of $\et{\mu\boxtimes\nu}$ with respect to $\et{\mu}$.
 \end{defn}
 Note that for $\mu,\nu\in\MT\cap\Mstar$, the principal subordination function of $\et{\mu\boxtimes\nu}$ with respect to $\et{\mu}$ is the usual subordination function.

 The following result might be obtained by approximation. We provide a direct proof.
 \begin{cor}\label{cor:3.13}
   Given $\mu\in\MT^k,\nu\in\MT\cap\Mstar$, let $\rho$ be the principal subordination distribution of $\et{\mu\boxtimes\nu}$ with respect to $\et{\mu}$, we have that
     \begin{equation}\nonumber
        \Sigma_{\rho}(z)=\Sigma_{\nu}(\et{\mu}(z)).
     \end{equation}
   In particular, if $\nu\in\mathcal{ID}(\boxtimes, \mathrm{T})$, we have $\rho\in\mathcal{ID}(\boxtimes, \mathrm{T})$.
 \end{cor}
 \begin{proof}
   By choosing a sequence $\mu_n\in\MT\cap\Mstar$ such that $\mu_n$ converges to $\mu$ weakly, Lemma \ref{subLemma} implies
     \begin{equation}\nonumber
        \et{\mu}(z)=\et{\mu\boxtimes\nu}(z\Sigma_{\nu}(\et{\mu}(z)))
     \end{equation}
 for $z$ in a neighborhood of zero.

  Set $\Phi(z)=z\Sigma_{\nu}(\et{\mu}(z))=z\cdot\mathcal{S}_{\nu}(\pp{\mu}(z))$, and we thus have
    \begin{equation}\label{for314}
       \et{\mu\boxtimes\nu}(z)=\et{\mu}(\omega_1(z))=\et{\mu\boxtimes\nu}(\Phi(\omega_1(z)))=\et{\mu\boxtimes\nu}(\Phi(\et{\rho}(z))).
    \end{equation}
Fix $1\leq j \leq k$, we claim that if $z\in g(D_{j,r})$, then $\Phi(\omega_1(z))=z$. Indeed, for $ 0\leq\arg(w^{1/k})< 2\pi/k $
  and $z=g(\omega^j w^{1/k})=\chi^{(j)}(w)$, using the construction of $\omega_1$ in Theorem \ref{BBnew}, we have
    \begin{equation}\label{eq:cor3.14a}
      \omega_1(z)=\iota_{\mu}^{(j)}(\et{\mu\boxtimes\nu}(z))=\chi_{\mu}^{(j)}(\pp{\mu\boxtimes\nu}(z)).
    \end{equation}
From (\ref{eq:3.10}) and (\ref{eq:3.12}), we have
      \begin{equation}\label{eq:cor3.14b}
        \chi^{(j)}(w)=g(\omega^j w^{1/k})), \text{and}\, \pp{\mu\boxtimes\nu}(\chi^{(j)}(w))=w.
      \end{equation}
Equations (\ref{eq:cor3.14a}) and (\ref{eq:cor3.14b}) imply that $\omega_1(g(\omega^{j} w^{1/k})))=\chi_{\mu}^{(j)}(w)$.

  Note that $\pp{\mu}(\omega_1(g(\omega^j w^{1/k})))=\pp{\mu\boxtimes\nu}(g(\omega^j w^{1/k}))=w$. Thus we obtain that
      \begin{equation}\nonumber
         \Phi(\omega_1(z))=\Phi(\omega_1(g(\omega^j w^{1/k})))=\chi_{\mu}^{(j)}(w)\mathcal{S}_{\nu}(w)=\chi^{(j)}(w)=z.
      \end{equation}
   The above claim, (\ref{for314}) and Proposition \ref{identityINV} imply
      \begin{equation}\nonumber
        z=\Phi(\omega_1(z))=\Phi(\et{\rho}(z))
      \end{equation}
   for $z\in g(D_{j,r})$. We conclude that $\Sigma_{\rho}(z)=\Phi(z)/z=\Sigma_{\nu}(\et{\mu}(z))$ for $z$ in a small neighborhood of zero by applying the above argument for all $1\leq j \leq k$.

  If $\nu\in\mathcal{ID}(\boxtimes,\mathrm{T})$, then by Theorem 6.7 in \cite{BV1993}, there exists an analytic function $u(z)$ defined in $\mathbb{D}$ such that $\Sigma_{\nu}(z)=\exp(u(z))$ and $\Re{u(z)}\geq 0$ for all $z\in\mathbb{D}$. Thus $\Sigma_{\nu}(\et{\mu}(z))=\exp(u(\et{\mu}(z)))$ and $\Re(u(\et{\mu}(z)))\geq 0$ for all $z\in \mathbb{D}$, and then the second assertion follows from Theorem 6.7 in \cite{BV1993}.
 \end{proof}
 \begin{rmk}
   \emph{If $k=1$, noticing that $\MT^k=\MT\cap\Mstar$, the modified $\mathcal{S}$-transform is the usual $\mathcal{S}$-transform. We see that Corollary 3.14 holds when $\mu,\nu\in\MT\cap\Mstar$. }
 \end{rmk}

The following result is the multiplicative analogue of Lemma 2.6.
\begin{prop}\label{prop:3.14}
 Given $\rho,\tau\in\MT\cap\Mstar$,
 let $\sigma$ be a measure in $\MT\cap\Mstar$ such that $\et{\rho\boxtimes\tau}(z)=\et{\rho}(\et{\sigma}(z))$. If $\sigma\in\mathcal{ID}(\boxtimes, \mathrm{T})$,
 then $\rho\boxtimes\tau^{\boxtimes t}$ can be defined for all $t\geq 0$ in the sense that $\Sigma_{\rho\boxtimes(\tau^{\boxtimes t})}(z)=\Sigma_{\rho}(z)(\Sigma_{\tau}(z))^t$.
\end{prop}
\begin{proof}
 For $t>0$, there exists $\mu_t\in\MT\cap\Mstar$ such that
     \begin{equation}\nonumber
        \et{\mu_t}(z)=\et{\rho}(\et{\sigma^{\boxtimes t}}).
     \end{equation}
Using a similar argument as in the proof of Lemma 2.6 and applying Corollary \ref{cor:3.13}, we can find that
 $\Sigma_{\mu_t}(z)=\Sigma_{\rho}(z)(\Sigma_{\tau}(z))^t$.
\end{proof}
\subsection{Semigroups related to multiplicative free convolution}
Recall that $\mathcal{M}_{\mathrm{T}}^*=\{\mu\in\MT\cap \Mstar:
  \eta_{\mu}(z)\neq 0, \forall z\in\mathbb{D}\backslash \{0\} \}$.
Given $\mu\in\mathcal{M}_{\mathrm{T}}^* $ and $t>1$, and
let $u$ be an analytic function satisfying that $z/(\et{\mu}(z))=e^{u(z)}$ holds for $z$ in a neighborhood of zero.
Set $H_t(z)=z e^{(t-1)u(z)}=z [z/(\et{\mu}(z))]^{t-1}$.
It is shown in
\cite{BB2005} that $H_t$ has a right inverse $\omega_t: \mathbb{D}\rightarrow\mathbb{D}$
such that $H_t(\omega(z))=z$,
and there exists a probability measure
$\multt{\mu}\in \MstarS$ such that
\begin{enumerate}[$(1)$]
  \item $\et{\mu^{\boxtimes t}}(z)=\et{\mu}(\omega_t(z))$
   and $\Sigma_{\multt{\mu}}(z)=\left( \Sigma_{\mu}(z) \right)^t$,
  \item $\omega_t(z)=\eta_{\multt{\mu}}(z) \left[z/\eta_{\multt{\mu}}(z) \right]^{1/t}$ for
  $z \in \mathbb{D}$, where the power is chosen such that the equation holds.
\end{enumerate}

Observe that for each $t>0$, by Proposition \ref{etaD}, there exists
a probability measure $\sigma_t \in \MT$ such that $\eta_{\sigma_t}(z)=\omega_{t+1}(z)$.
It turns out that $\sigma_t$ is $\boxtimes$-infinitely divisible and its
$\Sigma$-transform is $\Sigma_{\sigma_t}(z)=
\left[z/\eta_{\mu}(z) \right]^{t}$, which can be obtained by applying the same argument as in the proof of Lemma
\ref{rhoInf}.

The following result is a partial converse of Theorem 3.5 in \cite{BB2005}.
\begin{thm}
 Given $\mu\in\MT\cap\Mstar$, assume that for any $t>1$, there exists a
 probability measure $\mu_t\in\MT$ such that
     \begin{equation}\label{Massup}
       \Sigma_{\mu_t}(z)=\left(\Sigma_{\mu}(z) \right)^t.
     \end{equation}
 Assume in addition that $\mu_t$ is subordinated with respect to $\mu$ for all $t>1$.
 Then $\eta_{\mu}(z)\neq 0$  for all $z\in\mathbb{D}\backslash \{0\}$, that is $\mu\in\MstarS$.
\end{thm}
\begin{proof}
For each $t>1$,  we denote by $\omega_t$ the subordination function of $\mu_t$ to $\mu$. Observing that $\mu_t\in \Mstar$ and $\omega_t'(0)\neq 0$,
for each $t>1$, there exists a probability measure $\sigma_{t-1}\in\MT\cap\Mstar$ such that
     $\eta_{\sigma_{t-1}}(z)=\omega_t(z)$.
We rewrite (\ref{Massup}) as
   \begin{equation}\label{rewrite}
     \frac{\eta_{\mu_t}^{-1}(z)}{z}=\left[\frac{\eta_{\mu}^{-1}(z)}{z} \right]^t
   \end{equation}
for $z$ in a neighborhood of zero.

Note that $\omega_t^{-1}(z)=\eta_{\mu_t}^{-1}(\eta_{\mu}(z))$ for $z$ in a neighborhood
of zero.
Replacing $z$ by $\eta_{\mu}(z)$ in (\ref{rewrite}),
we obtain that
    \begin{align*}
      \frac{\omega_t^{-1}(z)}{\eta_{\mu}(z)}&=
          \frac{\eta_{\mu_t}^{-1}(\eta_{\mu}(z))}{\eta_{\mu}(z)} \\
         & =\left[\frac{\eta_{\mu}^{-1}(\eta_{\mu}(z))}{\eta_{\mu}(z)} \right]^t \\
         &=\left[\frac{z}{\eta_{\mu}(z)}\right]^t,
    \end{align*}
which implies that,
      \begin{equation}\nonumber
        \frac{\omega_t^{-1}(z)}{z}=\left[\frac{z}{\eta_{\mu}(z)}\right]^{t-1}.
      \end{equation}
Given $t>0$, we thus have $\Sigma_{\sigma_{t}}(z)=[z/\eta_{\mu}(z)]^t$
for $z$ in a neighborhood of zero.
Therefore $\sigma_t$ is $\boxtimes$-infinitely divisible.

By Theorem 6.7 in \cite{BV1992}, there exists
an analytic function $u(z)$ in $\mathbb{D}$ such that
$\Re{u(z)}\geq 0$ if $z\in\mathbb{D}$ and $\Sigma_{\sigma_1}(z)=\exp(u(z))$.
We thus obtain that
$z/\eta_{\mu}(z)=\exp(u(z))$, which implies that $\eta_{\mu}(z)\neq 0 $ for all $z\in\mathbb{D}$.
\end{proof}

It was pointed in \cite{BB2005} that $\multt{\mu}$ is only
determined up to a rotation by a multiple of $2\pi t$.
Note that $\omega_t$ and $\sigma_t$ are determined by the choice of $\multt{\mu}$.
\subsection{Multiplicative Boolean convolution and the
Bercovici-Pata bijection}
Multiplicative Boolean convolution on $\mathrm{T}$ was studied by Franz
\cite{Franz}.
Let $\mu\in\MT$, and we set $k_{\mu}(z)=z/\eta_{\mu}(z)$. Given
two probability measures $\mu,\nu\in\MT$, their
multiplicative Boolean convolution $\mu \bmultg \nu $
is a probability measure on $\mathrm{T}$ such that
\begin{equation}\nonumber
  k_{\mu\bmultk\nu}(z)=k_{\mu}(z)k_{\nu}(z)
\end{equation}
for all $z\in\mathbb{D}$.

A probability $\mu\in\MT$ is said to be $\bmultg$-infinitely divisible,
if for any positive integer $n$, there exists $\mu_n \in\MT$ such
that $\mu=(\mu_n)^{\bmultk n}$. Let $P_0$ be the Haar measure. It
is shown in \cite{Franz} that $\mu\in\MT\backslash \{P_0\}$
is $\bmultg$-infinitely divisible if and only if $\eta_{\mu}'(0)\neq 0$ and
$\eta_{\mu} \neq 0 $ for all $z\in\mathbb{D}\backslash\{0\}$, that is $\mu\in\MstarS$,
which is equivalent to
  \begin{equation}\label{levyMBool}
    k_{\mu}(z)=\exp\left(bi+\int_{\mathrm{T}}\frac{1+\xi z}{1-\xi z}d\,\tau_{\mu}(\xi) \right),
  \end{equation}
where $b\in\mathbb{R}$ and $\tau_{\mu}$ is a finite measure on $\mathrm{T}$.
Equation (\ref{levyMBool}) is the analogue of the
L\'{e}vy-Hin\v{c}in formula in this context.

The multiplicative Bercovici-Pata bijection from $\bmultg$ to $\boxtimes$
was studied in \cite{Wang}.
Denote the
set of all $\bmultg$-infinitely divisible measures on $\mathrm{T}$
by $\mathcal{ID}(\bmultg, \mathrm{T})$,
and the multiplicative Bercovici-Pata bijection from from $\bmultg$ to $\boxtimes$
by $\mathbb{M}$. Then we have $k_{\mu}(z)=\Sigma_{\mathbb{M}(\mu)}(z)$.

Given $\mu\in \mathcal{ID}(\bmultg, \mathrm{T})\backslash P_0=\MstarS$,
let $\omega_2$ be the subordination function of $\multtT{2}$ with
respect to $\mu$, and let $\sigma$ be the probability measure on $\mathrm{T}$
such that $\eta_{\sigma}(z)=\omega_2(z)$. Then
$\sigma$ is $\boxtimes$-infinitely divisible and its $\Sigma$-transform
is $\Sigma_{\sigma}(z)=z/\eta_{\mu}(z)=k_{\mu}(z)$. Therefore,
$\sigma$ is the same as $\mathbb{M}(\mu)$.
Since $P_0\boxtimes P_0=P_0$ and $\et{\tiny{P_0}}=z$, the subordination funciton of $P_0\boxtimes P_0$ with respect to
$P_0$ is the identity map $z$, and the measure associated with the identity map $z$ is $P_0$.
To summarize, we have
the following corollary.
\begin{cor}
 Given $\mu\in \mathcal{ID}(\bmultg, \mathrm{T})$,
 let $\omega_2$ be the subordination function of $\multtT{2}$ with
respect to $\mu$, and let $\sigma$ be the probability measure on $\mathrm{T}$
such that $\eta_{\sigma}(z)=\omega_2(z)$.
 Then $\sigma=\mathbb{M}(\mu)$, where $\mathbb{M}$ is the
 multiplicative Bercovici-Pata bijection from from $\bmultg$ to $\boxtimes$.
\end{cor}

\begin{prop}
 Let $\mu\in\MT$, then the following are equivalent.
   \begin{enumerate}[$(1)$]
     \item $\mu\in\mathcal{ID}(\mathrm{T}, \boxtimes)$;
     \item $\delta_{\beta}\bmultk \mu \in\mathcal{ID}(\mathrm{T}, \boxtimes)$ for any $\beta\in\mathrm{T}$.
   \end{enumerate}
\end{prop}
\begin{proof}
 It is enough to prove that (1) implies (2) for $\mu\in\MT\cap\Mstar$. Observing that $\et{\delta_{\beta} \bmultk \mu}(z)=\beta\cdot \et{\mu}(z)$, we thus have
    \begin{equation}\nonumber
      \Sigma_{\delta_{\beta}\bmultk\mu}(z)=\frac{\eta_{\delta_{\beta}\bmultk\mu}^{-1}(z)}{z}=\frac{\eta_{\mu}^{-1}(\overline{\beta}z)}{z}=\overline{\beta}\cdot\Sigma_{\mu}(\overline{\beta}z).
    \end{equation}
  The result follows from the L\'{e}vy-Hin\v{c}in
formula for the multiplicative free convolution on $\mathrm{T}$.
\end{proof}
\subsection{An analogue of equations studied by Belinschi and Nica}
In this subsection, we prove our Theorem \ref{thm:1.1}.
Recall that $\lambda_t$ is the free multiplicative analogue of the normal distribution on $\mathrm{T}$, the unit circle of $\mathbb{C}$,
with $\Sigma_{\lambda_t}(z)=\exp\left(\frac{t}{2}\frac{1+z}{1-z}\right)$ and we set $\lambda=\lambda_1$.
For $\mu\in\MT$, we denote $m_1(\mu)= \int_{\mathrm{T}}\xi d\mu(\xi)$.
\begin{prop}\label{prop:3.19}
 Given $\mu\in\MstarS$, and an analytic map $u=u(z)$ defined by
    \begin{equation}\label{Herglotz}
      u(z)=bi+ \int_{\mathrm{T}}\frac{1+\xi z}{1-\xi z}d\tau(\xi), \, z\in\mathbb{D},
    \end{equation}
where $b\in [0,2\pi)$ and $\tau$ is a finite measure on $\mathrm{T}$.
If $k_{\mu}(z)=z/\et{\mu}(z)=\exp(u(z))$, then
$b=\arg(1/m_1(\mu))\in [0,2\pi)$, and
 $\tau(\mathrm{T})=\ln|1/m_1(\mu)|$.
 In particular, there exists a probability measure $\nu\in \MT$ such that
 $k_{\mu}(z)=\Sigma_{\lambda}(\eta_{\nu}(z))$
 if and only if $m_1(\mu)=e^{-1/2}$.
\end{prop}
\begin{proof}
By definition, we have that
   \begin{equation}
      k_{\mu}(0)=\lim_{z\rightarrow 0}\frac{z}{\et{\mu}(z)}=\frac{1}{\et{\mu}'(0)}=\frac{1}{m_1(\mu)}.
   \end{equation}
Since $u(0)=b i + \tau(\mathrm{T})$, we obtain that
    \begin{equation}
    b=\arg\left(\frac{1}{m_1(\mu)} \right),\,\text{and}\, \tau(\mathrm{T})=\ln \left|\frac{1}{m_1(\mu)}\right|.
    \end{equation}
The first assertion follows.

By (\ref{formula}), we have that
   \begin{equation}\nonumber
         \Sigma_{\lambda}(\eta_{\nu}(z))=
   \exp\left( \frac{1}{2} \int_{\mathrm{T}}\frac{1+\xi z}{1-\xi z} d\nu(\xi)\right).
    \end{equation}
Noticing that $k_{\mu}$ has the Herglotz representation as (\ref{levyMBool}),
we conclude that $k_{\mu}(z)$ can be written in the form of $\Sigma_{\lambda}(\eta_{\nu}(z))$
for a probability measure $\nu$ on $\mathrm{T}$
if and only if $\ln (1/m_1(\mu))=1/2$.
This implies the second half of the assertion.
\end{proof}

For $\mu\in\mathcal{ID}(\bmultg, \mathrm{T})\backslash P_0= \MstarS$ with $m_1(\mu)>0$,
let $u(z)$ be the analytic function satisfying $k_{\mu}(z)=\exp(u(z))$ and $u(0)>0$.
Given $t>1$, let $H_t(z)=z \exp((t-1)u(z))$, and denote its right inverse by $\omega_t:\mathbb{D}\rightarrow \mathbb{D}$ with $\omega_t(0)=0$. We define (see \cite{BB2005}) $\mu^{\boxtimes t}$ by the relation
    \begin{equation}
      \et{\mu^{\boxtimes t}}(z)=\et{\mu}(\omega_t(z)).
    \end{equation}
Then we see that $H_t'(0)>0$, $\omega_t'(0)>0$
and that
 $$m_1(\mu^{\boxtimes t})=\et{\mu^{\boxtimes t}}'(0)>0.$$
For $t>0$, we also define $\mu^{\bmultk t}$ by the relation
       \begin{equation}
          k_{\mu^{\bmultk t}}(z)=\exp(tu(z)).
       \end{equation}
For this choice of the Boolean convolution power, we have that
    \begin{equation}\nonumber
       m_1(\mu^{\bmultk t})>0.
    \end{equation}
\begin{defn}
 Given $\mu\in\MstarS$ such that
 $m_1(\mu)>0$,
 we define a family of maps $\{ \mathbb{M}_t\}_{t\neq 0}$ by
    \begin{equation}\nonumber
      \mathbb{M}_t(\mu)=\left(\mu^{\boxtimes (t+1)} \right)^{\bmultk \frac{1}{t+1}},
    \end{equation}
  where we choose $\mu^{\boxtimes (t+1)}$ and $\mathbb{M}_t(\mu)$ in a way such that they have positive means.
\end{defn}
The next result is a special case of Theorem 4.4 in \cite{AHasebe2012}.
\begin{lemma}
Given $\mu\in\MstarS$ with
 $m_1(\mu)>0$, 
 then the following assertions are true.
\begin{enumerate}[$(1)$]
  \item $\mathbb{M}_{t+s}(\mu)=\mathbb{M}_t\left(\mathbb{M}_s(\mu) \right)$ for all $t,s\geq 0$.
  \item $\mathbb{M}_1(\mu)=\mathbb{M}(\mu)$.
\end{enumerate}
\end{lemma}

We are now able to prove Theorem \ref{thm:1.1}.
\begin{proof}[\textbf{Proof of Theorem \ref{thm:1.1}}]
We set
\begin{equation}
    u(z)=\frac{1}{2}\int_{\mathrm{T}}\frac{1+\xi z}{1-\xi z} d\nu(\xi),
\end{equation}
then by (\ref{formula}) and the assumption (\ref{equalM}), we have that
    \begin{equation}
       \frac{z}{\et{\mu}(z)}=\Sigma_{\lambda}(\et{\nu}(z))=\exp(u(z)).
    \end{equation}
By Proposition \ref{prop:3.19}, we see that $m_1(\mu)>0$. 
We therefore can
choose the multiplicative convolution power $\mu^{\boxtimes (t+1)}$ such that $m_1(\mu^{\boxtimes (t+1)})>0$.

Let $\eta_t$ be the the principal subordination function
of $\nu\boxtimes \lambda_t$ with respect to $\nu$ and
$\omega_{t+1}$ be the subordination function of $\multtT{(t+1)}$
with respect to $\mu$. Let $\rho_t, \sigma_t\in\MT$ such that $\et{\rho_t}=\eta_t$ and $\et{\sigma_t}=\omega_{t+1}$.

By Corollary \ref{cor:3.13}, (\ref{equalM})
implies that $\Sigma_{\rho_t}(z)=\Sigma_{\lambda_t}(\et{\mu}(z))=\exp(t u(z))$. From the choice
of $\mu^{\boxtimes (t+1)}$, the function $H_{t+1}(z):=z \exp(tu(z))$ is the left inverse of $\omega_{t+1}$
such that $H_{t+1}(\omega_{t+1}(z))=z$ for all $z\in\mathbb{D}$, which implies that
    \begin{equation}
       \Sigma_{\sigma_t}(z)=\exp(tu(z)).
    \end{equation}
We thus obtain that $\rho_t=\sigma_t$ and $\eta_t=\omega_{t+1}$.

Replacing $z$ by $\eta_t$ in (\ref{equalM}),
we obtain that
 \begin{equation}
   \begin{split}
     \Sigma_{\lambda}(\eta_{\nu\boxtimes\lambda_t}(z))&=
     \Sigma_{\lambda}(\eta_{\nu}((\eta_t(z))))
      =\frac{\eta_t(z)}{\eta_{\mu}(\eta_t(z))}\\
       &=\frac{\omega_{t+1}(z)}{\eta_{\mu}(\omega_{t+1}(z))}
       =\frac{\omega_{t+1}(z)}{\eta_{\multtT{(t+1)}}(z))} \\
       &=\left( \frac{z}{\eta_{\multtT{(t+1)}}(z)}\right)^{1/t+1}.\\
 \end{split}
 \end{equation}
On the other hand, by the definition of $\mathbb{M}_t$, we have that
 \begin{align*}
   \frac{z}{\eta_{\mathbb{M}_t(\mu)}(z)}&=
    \frac{z}{\eta_{\left(\mu^{\boxtimes (t+1)} \right)^{\bmultk \frac{1}{t+1}}}(z)}\\
     &=\left( \frac{z}{\eta_{\multtT{(t+1)}}(z)} \right)^{1/t+1}.
 \end{align*}
We thus complete the proof of Theorem \ref{thm:1.1}.
\end{proof}
\subsection{Some examples and applications}
We start with some examples 
which are the multiplicative analogue of examples studied
in \cite{Anshelevich,Anshelevich2012,AHasebe2012,BN2008Ind}. 
We define the set
\begin{equation}\nonumber
  (\mathrm{A})=\{\mu\in\MstarS: m_1(\mu)=e^{-1/2}\}.
\end{equation}
By Proposition \ref{Herglotz},
the set $\MT$ is in one-to-one correspondence with the set $(\mathrm{A})$ via
the bijection $\nu\leftrightarrow \mu$, such that
$\Sigma_{\lambda}\left(\eta_{\nu}(z) \right)=z/\eta_{\mu}(z).$
\begin{defn}
 The bijective map $\Lambda : \MT\rightarrow (\mathrm{A})$
  is defined by
  \begin{equation}\nonumber
    \Sigma_{\lambda}\left(\eta_{\nu}(z) \right)=\frac{z}{\eta_{\Lambda[\nu]}(z)}, \,\forall \nu\in\MT.
  \end{equation}
\end{defn}
Using the notation $\Lambda$, Theorem \ref{thm:1.1} implies that
\begin{equation}\nonumber
  \Lambda[\nu\boxtimes\lambda_t]=\mathbb{M}_t[\Lambda(\nu)],\,\,\forall \nu\in\MT.
\end{equation}

\begin{example}
\emph{
Let $\delta_1$ be the Dirac measure at 1, and let $\mu=\Lambda[\delta_1]$,
we have $z/\eta_{\mu}(z)=\Sigma_{\lambda}\left(\eta_{\delta_1}(z) \right)
  =\exp(\frac{1}{2}\frac{1+z}{1-z})$. For $t\geq 0$, Theorem \ref{thm:1.1} implies that
    \begin{equation}\nonumber
      \frac{z}{\eta_{\mathbb{M}_t(\mu)}(z)}
       =\Sigma_{\lambda}\left(\eta_{\delta_1\boxtimes\lambda_t}(z)\right)
        =\Sigma_{\lambda}\left(\eta_{\lambda_t}(z)\right).
    \end{equation}
In particular, when $t=1$,
   \begin{equation}\nonumber
      \eta_{\mathbb{M}_1(\mu)}(z)=
        \frac{z}{\Sigma_{\lambda}\left(\eta_{\lambda_1}(z) \right)}
         =\eta_{\lambda_1}(z),
   \end{equation}
   where we used the equality $\left( (z\Sigma_{\lambda})\circ \eta_{\lambda}\right)(z)=z$
   and $\lambda=\lambda_1$. Therefore, $\mathbb{M}_1(\mu)$ is the free multiplicative analogue of the normal distribution on $\mathrm{T}$. }
\end{example}

\begin{example}
  \emph{
   More generally, we consider $\lambda_{b,t}=\delta_{b}\boxtimes\lambda_t$
    and $\mu_{b,t}=\Lambda[\lambda_{b,t}]$.
   Then we have that
     \begin{equation}\nonumber
       \Sigma_{\lambda}\left(\eta_{\lambda_{b,t}}(z)\right)
        =\frac{z}{\eta_{\mu_{b,t}}(z)},\,\text{for}\, t \neq 0.
     \end{equation}
   On the other hand, Theorem \ref{thm:1.1} implies that
     \begin{equation}\nonumber
       \Sigma_{\lambda}\left(\eta_{\lambda_{b,t_1+t_2}}(z) \right)
          =\Sigma_{\lambda}\left(\eta_{\lambda_{b,t_1}\boxtimes\lambda_2}(z) \right)
          =\frac{z}{\eta_{\mathbb{M}_{t_2}(\mu_{b,t_1})}(z)}, \,\text{for}\, t_1,t_2\geq 0,
     \end{equation}
     which yields that $\mu_{b,t_1+t_2}=\mathbb{M}_{t_2}(\mu_{b,t_1})$ \,\text{for}\, $t_1, t_2 \geq 0$.  }
\end{example}
We would like to provide another example which covers part of Example 4.10 in \cite{AHasebe2012}.
\begin{example}
  \emph{
  Let $P_0$ be the Haar measure on $\mathrm{T}$. Then by the free independence,
    $P_0\boxtimes\lambda_t=P_0$. We set $\mu=\Lambda[P_0]$, and we have
    \begin{equation}\nonumber
      \Sigma_{\lambda}\left(\eta_{P_0\boxtimes\lambda_t}(z) \right)
      =\Sigma_{\lambda}\left(\eta_{P_0}(z) \right)=\frac{z}{\eta_{\mu}(z)},
    \end{equation}
   which implies that $\mathbb{M}_t(\mu)=\mu$ for all $t\geq 0$.
  To calculate the distribution of $\mu$, we note that $\eta_{P_0}\equiv 0$,
  which shows that $\eta_{\mu}=e^{-1}z$, and we thus have that $\psi_{\mu}(z)=z/e-z $.
  Using the identity
    \begin{equation}\nonumber
      \frac{1}{\pi}\left(\psi_{\mu}(z)+\frac{1}{2} \right)=
        \frac{1}{2\pi}\int_0^{2\pi}\frac{e^{it}+z}{e^{it}-z}d\,\mu(e^{-it}),
    \end{equation}
 and Stieltjes's inversion formula, we obtain that
    \begin{equation}\nonumber
      \mu(dt)=\frac{1}{2\pi}\frac{1-e^{-2}}{1+e^{-2}-2e^{-1} \cos(t)}dt,\,\, 0\leq t \leq 2\pi.
    \end{equation}  }
\end{example}

We then give some applications of results concerning infinity divisibility of the measures associated with
subordination functions. For $\mu\in\MT$, we say $\mu$ is nontrivial if it is not a Dirac measure at a point
on $\mathrm{T}$.
\begin{lemma}\label{lemma3.25}
 Given $\sigma\in\mathcal{ID}(\boxtimes,\mathrm{T})$ which is non-trivial, and $0<\epsilon<1$, there exists a positive
 number $n(\epsilon)$ such that
    \begin{equation}\nonumber
            \et{\sigma_t}(\overline{\mathbb{D}})\subset \mathbb{D}_{\epsilon}=\{z=re^{i\theta}:0\leq r\leq\epsilon, 0\leq \theta<2\pi \}
    \end{equation}
holds for any $t>n(\epsilon)$, where $\sigma_t=\sigma^{\boxtimes t}$.
\end{lemma}
\begin{proof}
 If $\sigma=P_0$, the Haar measure on $\mathrm{T}$, then the result is trivial. If $\sigma \neq P_0$ is nontrivial, then
 by Theorem 6.7 in \cite{BV1992}, there exists a finite positive measure $\nu$ on $\mathrm{T}$ with $\nu(\mathrm{T})>0$,
 $\alpha\in\mathbb{R}$, and an analytic function $u$ defined by
     \begin{equation}\nonumber
       u(z)=i\alpha +\int_{\mathrm{T}}\frac{1+\xi z}{1-\xi z}d\nu(\xi), \,\text{for}\,z\in\mathbb{D},
     \end{equation}
 such that $\Sigma_{\sigma}(z)=\exp(u(z))$. We choose $\sigma_t\in\MT$ satisfying $\Sigma_{\sigma_t}(z)=\exp(tu(z))$. Noticing that other choices of the multiplicative free convolution power of $\sigma$ can be obtained from $\sigma_t$ by a rotation of
 a multiple of $2\pi t$, it is enough to prove the assertion for $\sigma_t$.

 We set $\Phi_{\sigma_t}=z\Sigma_{\sigma_t}(z)$, then by Lemma \ref{infMstar}, we have that
       \begin{equation}\nonumber
          \Phi_{\sigma_t}^{-1}(\mathbb{D})=\et{\sigma_t}(\mathbb{D}).
       \end{equation}
 For $z=re^{i\theta}\in\mathbb{D}$, we calculate
      \begin{equation}\label{eq:3.33}
        \begin{split}
          |\Phi_{\sigma_t}(z)|&=r\exp\left(t\int_{\mathrm{T}}\frac{1-r^2}{|1-\xi z|^2}d\nu(\xi) \right)\\
                      &\geq r\exp\left(t \int_{\mathrm{T}}\frac{1-r^2}{|1+r|^2}d\nu(\xi) \right)\\
                      &=r \exp\left(t\cdot \nu(\mathrm{T})\frac{1-r}{1+r} \right).
        \end{split}
      \end{equation}
 Since  $\lim_{t\rightarrow\infty} r\exp\left( t\cdot\nu(\mathrm{T})\frac{1-r}{1+r} \right)=\infty$, we deduce that for any $0<\epsilon<1$, there exists a positive number $n(\epsilon)$ such that, for all $t>n(\epsilon)$, we have that
        \begin{equation}\nonumber
           |\Phi_{\sigma_t}(z)|>1, \,\text{for}\, |z|=\epsilon.
        \end{equation}
By Lemma \ref{infMstar}, $\Phi_{\sigma_t}(\mathbb{D})$ is a simply connected domain which contains zero, which implies that
       \begin{equation}\nonumber
         \et{\sigma_t}(\mathbb{D})=\Phi_{\sigma_t}^{-1}(\mathbb{D})\subset\mathbb{D}_{\epsilon}, \,\text{for}\, t>n(\epsilon).
       \end{equation}
The assertion follows by the fact that $\et{\sigma_t}$ extends to be a continuous function on $\overline{\mathbb{D}}$.
\end{proof}
For $\mu\in\MT$, we have that
   \begin{equation}\nonumber
      \frac{1}{2\pi}\left(\frac{1+\et{\mu}(z)}{1-\et{\mu}(z)}\right)=\frac{1}{2\pi}
       \int^{2\pi}_{0}\frac{e^{i\theta}+z}{e^{i\theta}-z}d\mu(e^{-i\theta}), \,z\in\mathbb{D}.
   \end{equation}
The real part of this function is the Poisson integral of the measure $d\mu(e^{-i\theta})$,
we can recover $\mu$ by Stieltjes's inversion formula. The functions
   \begin{equation}\label{realpartM}
     \frac{1}{2\pi}\Re{\left(\frac{1+\et{\mu}(re^{i\theta})}{1-\et{\mu}(re^{i\theta})} \right)}
       =\frac{1}{2\pi}\frac{1-|\et{\mu}(re^{i\theta})|^2}{|1-\et{\mu}(re^{i\theta})|^2}
   \end{equation}
converge to the density of $\mu(e^{-i\theta})$ a.e. relative to Lebesgue measure, and they converge to infinity a.e. relative to the singular part of this measure.
\begin{prop}\label{prop:3.26}
 Given $\mu\in\MT$ and $\sigma\in\mathcal{ID}(\boxtimes, \mathrm{T})$ which is nontrivial, let $\mu_t$ be the unique probability measure on $\mathrm{T}$ such that
       \begin{equation}\nonumber
          \et{\mu_t}(z)=\et{\mu}(\et{\sigma_t}(z)).
       \end{equation}
Then we have
     \begin{equation}\nonumber
        \lim_{t\rightarrow \infty} \sup_{\theta\in [0,2\pi]}\left| \frac{d\mu_t(e^{i\theta})}{d\theta}- \frac{1}{2\pi}\right|=0,
     \end{equation}
 where $d\mu_t(e^{i\theta})/d\theta$ is the density function of $\mu_t$ at $e^{i\theta}$ with respect to
 Lebesgue measure.
 \begin{proof}
   Given $0<\epsilon<1$, by Lemma \ref{lemma3.25}, there exists $n(\epsilon)>0$ such that $\et{\sigma_t}(e^{i\theta})<\epsilon$ for $t\geq n(\epsilon)$, which yields that $\et{\mu_t}(z)$ extends continuously to $\overline{\mathbb{D}}$. We thus have that
     \begin{equation}\nonumber
        |\et{\mu_t}(e^{i\theta})|=|\et{\mu}(\et{\sigma_t}(e^{i\theta}))|\leq |\et{\sigma_t}(e^{i\theta})|<\epsilon,
     \end{equation}
 which implies that
    \begin{equation}\label{eq:3.34}
      \frac{1-\epsilon}{1+\epsilon}=\frac{1-\epsilon^2}{|1+\epsilon|^2}\leq
         \frac{1-|\et{\mu_t}(e^{i\theta})|^2}{|1-\et{\mu_t}(e^{i\theta})|^2}\leq \frac{1}{|1-\epsilon|^2}.
    \end{equation}
Since $\epsilon$ is arbitrary, combining (\ref{realpartM}) with (\ref{eq:3.34}), we prove our assertion.
 \end{proof}
\end{prop}
\begin{cor}\label{cor:3.27}
 Given $\mu\in\MT$ and a nontrivial measure $\nu\in\mathcal{ID}(\boxtimes, \mathrm{T})$, then the density functions of the measures $\mu\boxtimes\nu_t$
 converge to $1/2\pi$ uniformly as $t\rightarrow \infty$; if $\mu\in\MT$ is nontrivial, then the density functions of the measures $\mu^{\boxtimes t}$ converge to $1/2\pi$ uniformly as $t\rightarrow\infty$.
\end{cor}
\begin{proof}
 Noticing Corollary \ref{cor:3.13}, Propositions \ref{prop:3.14}, \ref{prop:3.26} and Subsection 3.3, we only need to prove the case of $\mu^{\boxtimes t}$ for $\mu\notin \MstarS$. We point out that the measures are nontrivial imply that the subordination distributions involved are nontrivial.

For $\mu\in\MT\backslash\Mstar$, we have $\mu^{\boxtimes n}=P_0$, where $P_0$ is the Haar measure on $\mathrm{T}$.
Thus the assertion is true for this case. For $\mu\in\MT\cap\Mstar$, but $\mu\notin\MstarS$, it is shown in
\cite{BB2005} that $\mu\boxtimes\mu\in\MstarS$, thus this case reduces to the case when $\mu\in\MstarS$. This finishes the proof.
\end{proof}
\section{multiplicative convolution on $\Mrealplus$}
\subsection{Multiplicative free convolution on $\Mrealplus$}
We are interested in the probability measures
on the positive real line $\mathbb{R}^+$, which are different from the Dirac measure at zero, we thus
set
\begin{equation}\nonumber
  \Mrealpplus=\mathcal{M}_{\mathbb{R}^+}\backslash\{\delta_0\}.
\end{equation}
Given $\mu\in\Mrealpplus$, we define
\begin{equation}\nonumber
  \psi_{\mu}(z)=\int_0^{+\infty}\frac{tz}{1-tz}d\,\mu(t),
\end{equation}
and $\et{\mu}(z)=\pp{\mu}(z)/(1+\pp{\mu}(z))$.
The transform $\et{\mu}$ is characterized by the following proposition (see \cite{BB2005}).
\begin{prop}
Let $\eta:\mathbb{C}\backslash\mathbb{R}^+\rightarrow\mathbb{C}$ be an analytic function such that
$\eta(\overline{z})=\overline{\eta(z)}$ for all $z\in \mathbb{C}\backslash\mathbb{R}^+$.
Then the following two conditions are equivalent.
  \begin{enumerate}[$(1)$]
    \item $\eta=\et{\mu}$ for some $\mu\in\Mrealpplus$.
    \item $\eta(0-)=0$ and $\arg(\eta(z))\in [\arg z, \pi)$ for all $z\in\mathbb{C}^+$.
  \end{enumerate}
\end{prop}

It can be shown that $\eta_{\mu}$ is invertible in some neighborhood of $(-\infty,0)$,
and we set $\Sigma_{\mu}(z)=\eta_{\mu}^{-1}(z)/z$ where
$\et{\mu}^{-1}$ is defined in some neighborhood of $(\alpha,0)$.
Given two measures $\mu,\nu\in\Mrealpplus$, the multiplicative free convolution of
$\mu$ and $\nu$ is the probability measure $\mu\boxtimes\nu$ in $\Mrealpplus$
such that
 \begin{equation}\nonumber
   \Sigma_{\mu\boxtimes\nu}(z)=\Sigma_{\mu}(z)\Sigma_{\nu}(z)
 \end{equation}
holds in some neighborhood of $(\alpha,0)$, where these functions are defined.

It is known from \cite{BB2007new,Biane1998} that there exist two analytic functions
$\omega_1,\omega_2:\mathbb{C}\backslash\mathbb{R}^+ \rightarrow \mathbb{C}\backslash\mathbb{R}^+$
such that
\begin{enumerate}[$(1)$]
  \item $\omega_j(0-)=0$ for $j=1,2$,
  \item for any $\lambda\in\mathbb{C}^+$, we have $\omega_j(\overline{\lambda})=
      \overline{\omega_j(\lambda)}$ for $j=1,2$,
  \item $\et{\mu\boxtimes\nu}(z)=\eta_{\mu}\left(\omega_1(z) \right)
        =\eta_{\nu}\left(\omega_2(z)\right)$ for $z\in\mathbb{C}\backslash\mathbb{R}^+$.
\end{enumerate}
For simplicity, we say that $\omega_1$ (resp. $\omega_2$) is the subordination
function of $\mu\boxtimes\nu$ with respect to $\mu$ (resp. $\nu$), and
$\mu\boxtimes\nu$ is subordinated to $\mu$ and $\nu$.

The analogy of the L\'{e}vy-Hin\v{c}in in this setting was proved in
\cite{BV1992,BV1993}. Given $\mu\in\Mrealplus$, then $\mu\in\Mrealpplus$ is
$\boxtimes$-infinitely divisible if and only if $\Sigma_{\mu}(z)=\exp\left(u(z)\right)$,
with
  \begin{equation}\nonumber
    u(z)=a-bz+\int_0^{+\infty}\frac{1+tz}{z-t}d\,\sigma(t),
  \end{equation}
where $b\in\mathbb{R}$ and $\sigma$ is a finite positive measure on $\mathbb{R}^+$.
The analogue of the normal distribution in this context is given by
$\Sigma_{\lambda_t}(z)=\exp\left(\frac{t}{2} \frac{z+1}{z-1} \right)$.

\begin{lemma}\label{subLemmaB}
  Let $\mu,\nu\in\Mrealpplus$, we have
    \begin{equation}\nonumber
      \eta_{\mu}(z)=\eta_{\mu\boxtimes\nu}\left(z\Sigma_{\nu}(\eta_{\mu}(z)) \right)
    \end{equation}
holds in some neighborhood of interval $(\alpha,0)$.
\end{lemma}
The proof of Lemma \ref{subLemmaB} is identical to the proof of Lemma \ref{subLemma},
therefore we omit the details.

For any $t>0$, assume that $\eta_t:\mathbb{D}\rightarrow\mathbb{D}$ is the
subordination function of $\mu\boxtimes\lambda_t$ with respect to $\mu$,
by Lemma \ref{subLemmaB} and the characterization of $\eta$-transform,
there exists a probability measure $\rho_t$ in $\Mrealpplus$ such that $\eta_{\rho_t}(z)=\eta_t(z)$.
The argument in the proof of Lemma \ref{rhoInf} implies the following
result.
\begin{prop}
  The measure $\rho_t$ is $\boxtimes$-infinitely divisible and its $\Sigma$-transform is
   $\Sigma_{\rho_t}(z)=\Sigma_{\lambda_t}\left(\eta_{\mu}(z) \right)$,
    and
    \begin{equation}\nonumber
      \Sigma_{\lambda_t}\left(\eta_{\mu}(z) \right)
       =\exp \left( \frac{t}{2} \int_0^{+\infty}\frac{1+\xi z}{\xi z -1}d\,\mu(\xi) \right).
    \end{equation}
\end{prop}

We now discuss free convolution semigroups.
Given $t>1$, it is proved in \cite{BB2005} that
one can define $\multt{\mu} \in \Mrealpplus$ such
that $\Sigma_{\multt{\mu}}(z)=\left(\Sigma_{\mu}(z)\right)^t$ for
$z<0$ sufficiently close to zero.
Similar to the case of $\MT$, $\multt{\mu}$ is subordinated
with respect to $\mu$ and we denote the
subordination function by $\omega_t$.
By Theorem 2.6 in \cite{BB2005} and the characterization
of $\eta$-transform, there exists a probability $\sigma_t\in\Mrealpplus$
such that $\eta_{\sigma_t}(z)=\omega_{t+1}$ for all $t>0$.
Moreover, $\sigma_t$ is $\boxtimes$-infinitely divisible
and its $\Sigma$-transform is $\Sigma_{\sigma_t}(z)=[z/\eta_{\mu}(z)]^t$.
\subsection{Multiplicative Boolean convolution on $\Mrealplus$ and the semigroup
 $\mathbb{M}_t$}
Bercovici proved in \cite{Bhari2006} that
the multiplicative Boolean convolution does not preserve $\Mrealplus$.
But we can still define $\mu^{\bmultk t}$ for
$\mu\in\Mrealplus$ and $0\leq t \leq 1$ as follows.
Let $k_{\mu}(z)=z/\eta_{\mu}(z)$, the Boolean convolution power $\mu^{\bmultk t}$
is defined by
   \begin{equation}\nonumber
     k_{\mu^{\bmultk t}}(z)=\left(k_{\mu}(z)\right)^t.
   \end{equation}

The following definition was given in \cite{AHasebe2012}.
\begin{defn}
 A family of maps from $\Mrealplus$ to itself is defined by
     \begin{equation}\nonumber
       \mathbb{M}_t(\mu)=\left(\mu^{\boxtimes(t+1)} \right)^{\bmultk \frac{1}{t+1}}.
     \end{equation}
\end{defn}
It is also shown in \cite{AHasebe2012} that $\mathbb{M}_{t+s}=\mathbb{M}_t \circ \mathbb{M}_s$
for $t,s\geq 0$.
\subsection{Analogous equations}
Given a pair of probability measures $\nu,\mu\in \Mrealplus$,
we also consider, as in the case $\MT$, the semigroups
$\nu\boxtimes\lambda_t$ and $\mu^{\boxtimes (t+1)}$, the
subordination functions  $\eta_t$ and $\omega_{t+1}$, and
their associated probability measures $\rho_t$, $\sigma_t$ for all $t>0$.
Since $\Sigma_{\rho_t}(z)=\Sigma_{\lambda_t}(\et{\nu}(z))$ and $\Sigma_{\sigma_t}(z)=[z/\et{\mu}(z)]^t$,
we deduce that $\eta_t=\omega_{t+1}$ if and only if
   \begin{equation}\nonumber
     \Sigma_{\lambda}\left(\eta_{\nu}(z) \right)=\frac{z}{\eta_{\mu}(z)}.
   \end{equation}

Applying the same argument as in the proof of the Theorem \ref{thm:1.1}, we
obtain the following result.
\begin{thm}
   Given a pair of probability measures $\mu,\nu\in\Mrealplus$,
  such that
     \begin{equation}\label{equalMreal}
       \Sigma_{\lambda}(\eta_{\nu}(z))=\frac{z}{\eta_{\mu}(z)}, z\in \mathbb{C}^+.
     \end{equation}
   Then we have
     \begin{equation}\nonumber
       \Sigma_{\lambda}(\eta_{\nu\boxtimes\lambda_t}(z))=
       \frac{z}{\eta_{\mathbb{M}_t(\mu)}(z)}, z\in \mathbb{C}^+.
     \end{equation}
\end{thm}

\newsection{a description of the analogue of the normal distribution}
In \cite{BianeBM, BianeJFA}, Biane studied free Brownian motion and proved many important results. 
In this section, we give a new proof for the density functions of the free multiplicative analogue of the normal distributions,
which was first obtained in \cite{BianeJFA} (See also \cite{DH2011} for a different approach.).
Some results are new. For example, we show that $\lambda_t$ is unimodal for the circle case;
and we show that $\Phi_{\lambda}^{-1}(\mathbb{C}^+)$ contains infinitely many connected components where
$\lambda$ is the free multiplicative analogue of the normal distribution on the positive half line with
$\Sigma_{\lambda}(z)=\exp\left (\frac{z+1}{z-1}\right)$.
We also give a description of the boundaries $\Omega_t, \Omega$ (defined below), we observe that
$\partial\Omega_t$ can be parametrized by $\theta$ and $\partial\Omega$ can be parametrized by $r$.
\subsection{The circle case}
Let $\lambda_t\in\MT$ be the analogue of the normal distribution such that
$\Sigma_{\lambda_t}(z)=\exp(\frac{t}{2}\frac{1+z}{1-z})$.
We set $\Phi_t(z)=z\Sigma_{\lambda_t}(z)$, and let $\Omega_t=\{z\in\mathbb{D}:|\Phi_t(z)|<1 \}$. By Lemma \ref{infMstar},
$\et{\lambda_t}$ extends continuously to the unit circle $\mathrm{T}$,
$\Omega_t$ is simply connected and bounded by a simple closed curve, and we have that $\partial\Omega_t=\et{\lambda_t}(\mathrm{T}) $.

Observe that for $t\neq 4$, $\Phi_t$ has zeros of order one at
$z_1(t)=(2-t+\sqrt{t^2-4t})/2$ and $z_2(t)=(2-t-\sqrt{t^2-4t})/2$.
$\Phi_4$ has a zero of order two at $-1$; and for all $t$,
$\Phi_t$ has an essential singularity at $1$, and no other zeros and singularities.
For $0<t<4$, $z_1(t), z_2(t)\in\mathrm{T}$ and $z_2(t)=\overline{z_1(t)}$, we let $\theta_1(t)\in(0,\pi)$ and $\theta_2(t)\in(\pi,2\pi)$ such that $z_1(t)=e^{i\theta_1(t)}$ and $z_2(t)=e^{i\theta_2(t)}$. We have $z_1(4)=z_2(4)=-1$ and for $t>4$, $z_1(t)\in(-1,0)$ and $z_2(t)\in(-\infty, -1)$.

We define $$g_t(r,\theta)=r\exp\left(\frac{t}{2}\frac{1-r^2}{1-2r\cos\theta+r^2}\right)=|\Phi_t(z)|$$
 for $z=re^{i\theta}$. The unit circle is parametrized by $\mathrm{T}=\{e^{i\theta}: 0\leq\theta<2\pi\}$.
\begin{lemma}\label{symmetryMC}
 For $0<t<4$, $\partial\Omega_t=\{z=e^{i\theta}:\theta_1(t)\leq\theta\leq\theta_2(t)\}\cup \mathcal{L}_{1,t} \cup\mathcal{L}_{2,t}$,
 where $\mathcal{L}_{1,t}$ is an analytic curve, and $\mathcal{L}_{1,t}$ is in $\mathbb{D}\cap \mathbb{C}^+$ except one of its endpoints, and
 $\mathcal{L}_{2,t}$ is the reflection of $\mathcal{L}_{1,t}$ about $x$-axis.
 $\mathcal{L}_{1,t}$ can be parametrized by $\gamma_t(u)$ $(0\leq u \leq 1)$ such that $\gamma_t(0)\in\mathbb{R}$, $\gamma_t(1)=z_1(t)$ and $\gamma_t(u)\subset\mathbb{D}\cap\mathbb{C}^+$ for $0<u<1$. Moreover, $|\gamma_t(u)|$ is an increasing function of $u$ on the interval $[0,1]$.
\end{lemma}
\begin{proof}
 Observing that $\Phi_t(\overline{z})=\overline{\Phi_t(z)}$, we see that $\partial\Omega_t$ is symmetric with respect to $x$-axis. Since $\Omega_t$ is simply connected and $\partial\Omega_t$ is a simple closed curve, $\partial\Omega_t$ intersects $x$-axis at two points.

 Restricting $\Phi_t$ to real numbers, we find that $\Phi_t(\mathbb{R})\subset\mathbb{R}$, and that $\Phi_t$ is an increasing function on $(-1,1)$ since $\Phi_t'(z)$ is positive for $z\in(-1,1)$. From $\Phi_t(-1)=-1$ and $\lim_{z\rightarrow 1^-}\Phi_t(z)=+\infty$, we deduce that
 \begin{equation}\label{eq:5.1}
 \Phi_t^{-1}((-1,1))=(-1,x(t)),
 \end{equation}
  where $x(t)$ is the unique solution of the equation $\Phi_t(z)=1$ for $z\in(-1,1)$. The fact that $\Phi_t'(z)\neq 0$ for $z\neq z_t(t),z_2(t)$ implies that $\Phi_t$ is locally invertible for
 $z\neq z_1(t),z_2(t)$. Combining the fact that $\Phi_t(\mathrm{T}\backslash\{1\})\subset \mathrm{T}$, we obtain that
 $$\{e^{i\theta}: \theta_1(t)\leq \theta\leq \theta_2(t)\}\subset \partial \Omega_t$$
  and $\partial\Omega_t$ has corners of opening $\pi/2$ at $z_1(t)$ and $z_2(t)$.

 Since $\Phi_t$ is a conformal mapping from $\Omega_t$ to $\mathbb{D}$, by the symmetry $\Phi_t(\overline{z})=\overline{\Phi_t(z)}$ and (\ref{eq:5.1}),
 notcing that $\Phi_t'(0)=1$, we thus deduce that $\Phi_t(\Omega_t\cap\mathbb{C}^+)\subset \overline{\mathbb{D}}\cap\mathbb{C}^+$. $\partial\Omega_t$ is a simple closed curve, thus $z_1(t)$ and $x(t)$ are connected by $\partial\Omega_t$.
 It is clear that $\partial \Omega_t\backslash\{e^{i\theta}: \theta_1(t)\leq \theta\leq \theta_2(t)\}$ does not intersect with $\mathrm{T}$,
 we thus assume the curve $\gamma_t=\{\gamma_t(u):0\leq u\leq 1\}$ is the part of $\Omega_t$ which connects $z_1(t)$ and $x(t)$ such that $\gamma_t(0)=x(t)$, $\gamma_t(1)=z_1(t)$ and $\gamma_t(u)\in\mathbb{D}$ for $0<u<1$.

 We claim that $|\gamma_t(u)|$ is an increasing function of $u$ on the interval $[0,1]$. For given $0<r<1$,
 we define the function of $\theta$ by
   $$g_{t,r}(\theta)=g_t(r,\theta)=|\Phi_t(re^{i\theta})|.$$
Then $g_{t,r}$ is a strictly decreasing function of $\theta$ on the interval $[0,\pi]$. From the fact that $\Omega_t$ is simply connected, we deduce that, for $z_0\in\overline{\Omega_t}\cap\mathbb{D}\cap\mathbb{C}^+$, the arc
   \begin{equation}\label{eq:5.2}
     \{re^{i\theta}: |r|=|z_0|, \arg z_0 < \theta\leq \pi\} \subset\Omega_t.
   \end{equation}
Given $0<u_1<u_2<1$, we need to prove that $|\gamma_t(u_1)|<|\gamma_t(u_2)|$.
Since $[0,x(t)]\subset\overline{\Omega_t}$, we obtain from (\ref{eq:5.2}) that
  \begin{equation}
\{ re^{i\theta}:0\leq r \leq x(t), 0<\theta\leq \pi \}\subset \Omega_t,
  \end{equation}
which shows that $|\gamma_t(u_1)|>x(t)$.
Suppose that $|\gamma_t(u_1)|\geq |\gamma_t(u_2)|$, then there exists $0<u_1'\leq u_1$ such that $|\gamma_t(u_1')|=|\gamma_t(u_2)|$.
If $\arg(\gamma_t(u_1'))>\arg(\gamma_t(u_2))$, then by (\ref{eq:5.2}), $\gamma_t(u_2)\in\Omega_t$ and thus $\gamma_t(u_2)\notin\partial\Omega_t$; if $\arg(\gamma_t(u_1'))<\arg(\gamma_t(u_2))$, then $\gamma_t(u_1')\in\Omega_t$ and thus $\gamma_t(u_1')\notin\partial\Omega_t$. For both cases, we obtain a contradiction. Thus $|\gamma_t(u_1)|<|\gamma_t(u_2)|$ and our claim is proved.
\end{proof}

For $t>0$, we let $x_1(t)\in(0,1)$ be the unique solution of the equation $\Phi_t(z)=1$ for $z\in(0,1)$. For $0<t\leq 4$ we let $x_2(t)=-1$; for $t>4$, we let $x_2(t)\in(-1,0)$ be the unique solution of the equation $\Phi_t(z)=-1$ for $z\in(-1,0)$.

\begin{lemma}\label{symmetryMC00}
 For $t\geq 4$, $\partial\Omega_t=\mathcal{L}_{1,t}\cup\mathcal{L}_{2,t}$, where $\mathcal{L}_{1,t}$ is an analytic curve, and $\mathcal{L}_{1,t}$ is in $\mathbb{D}\cap \mathbb{C}^+$ except its endpoints, and $\mathcal{L}_{2,t}$ is the reflection of $\mathcal{L}_{1,t}$ about $x$-axis. $\mathcal{L}_{1,t}$ can be parametrized by $\gamma_t(u)$ $(0\leq u \leq 1)$ such that $\gamma_t(0)=x_1(t), \gamma_t(1)=x_2(t)$ and $\gamma_t(u)\subset\mathbb{D}\cap\mathbb{C}^+$ for $0<u<1$. Moreover, $|\gamma_t(u)|$ is an increasing function of $u$ on the interval $[0,1]$.
\end{lemma}
\begin{proof}
 Recall that $\Phi_4$ has a zero of order two at $-1$. For all $t>4$, $z_2(t)<-1$ and $z_1\in(-1,0)$.
 The assertion follows from the similar arguments in the proof Lemma \ref{symmetryMC}.
\end{proof}

From the proof of Lemmas \ref{symmetryMC} and \ref{symmetryMC00}, for $t>0$, we have that $\Phi_t^{-1}((-1,1))=(x_2(t),x_1(t))$. Moreover, $x_1(t)=\min\{|z|:z\in\partial\Omega_t \}$ and $-x_2(t)=\max\{|z|:z\in\partial\Omega_t\}$.

\begin{rmk}
 \emph{
 In fact, for any $t>0$, from the equation
    \begin{equation}\nonumber
        g_t(r,\theta)=0,\, 0<r<1, \,0\leq \theta\leq \pi,
    \end{equation}
we can prove that $dr/d\theta >0 $ for $0<\theta<\pi$, which implies that
if $z\in \partial\Omega_t$, then the entire radius $\{rz:0\leq r<1 \}$ is contained in $\Omega_t$. Therefore,
$\partial\Omega_t$ can be parametrized by $\theta$.}
\end{rmk}

\begin{lemma}\label{increasing}
Using the same notations in Lemmas \ref{symmetryMC} and \ref{symmetryMC00}, for $t>0$, the function $|1-\gamma_t(u)|$ is an increasing
function of $u$ on $[0,1]$.
\end{lemma}
\begin{proof}
We only prove the case when $0<t<4$, the proof for other cases are similar.
Noticing that $|1-re^{i\theta}|^2=1-2\cos\theta+r^2$, since $|\gamma_t(u)|$ is an increasing function of $u$, to prove the assertion, we only need to prove that for the implicit function
$r\exp(\frac{t}{2}\frac{1-r^2}{h})=1$ of $r$ and $h$, then $h$ increases
when $r$ increases on $(0,1)$. From this equation, we have
$h=h(r)=-(t/2)(1-r^2)/(\ln r)$. One can check that $h'(r)>0$ for $0<r<1$, therefore $h$ is an increasing function of $r$.
\end{proof}

\begin{thm}\label{supportU}
  Denote by $A_t$ the support of $\lambda_t$.
 \begin{enumerate}[$(1)$]
   \item For $t>0$, the measure $\lambda_t$ has no singular part, and its density function is an analytic function. $A_{t_1}\subset A_{t_2}$ if $t_1 < t_2<4$. $A_t \subsetneq \mathrm{T}$ for $0<t<4$ and $A_t = \mathrm{T}$ for $t\geq 4$.
   \item The measure $\lambda_t$ is unimodal for all $t>0$ and its density is maximal at $z=1$ and is minimal at $z=-1$.
   \item The density fucntion $d\lambda_t/d\theta$ converges uniformly to $1/(2\pi)$ as $t\rightarrow \infty$.
 \end{enumerate}
\end{thm}
\begin{proof}
 Since $z=1$ is not in the closure of $\Omega_t=\et{\lambda_t}(\mathbb{D})$, then $\lambda_t$ has no singular part. From the analyticity of $\Phi_t$ and or a general theorem in \cite{BB2005}, the density function is analytic.

For $0<t<4$, set $a_1(t)=\Phi_t(z_1(t)), a_2(t)=\Phi_t(z_2(t))$. Note that $\eta_{\lambda_t}(\Phi_t(z))=z$ for $z\in\overline{\Omega_t}$.
From (\ref{realpartM}) we see that $A_t$ is the closed arc on $\mathrm{T}$ with endpoints $a_1(t), a_2(t)$ which contains $1$. Thus, to prove that $A_{t_1}\subset A_{t_2}$,
it is enough to prove that $\arg(a_1(t))$ is an increasing function of $t$. A direct computation shows that $|z_1(t)-1|^2=t$ and $\arg(\Sigma_{\lambda_t}(z_1(t)))=\Im{z_1(t)}=\sqrt{t(4-t)}/2$.
We thus have
$$\arg(a_1(t))=\Im{z_1(t)}+\arg(z_1(t))= \sin(\theta_1(t))+\theta_1(t).$$
From  $z_1(t)=(2-t+\sqrt{t^2-4t})/2$ we see that $\theta_1(t)$ is an increasing function of $t$. The function $\theta\rightarrow \sin(\theta)+\theta$ is an increasing function on $(0,\pi)$. Thus $\arg(a_1(t))$ is an increasing function of $t$ and (1) is proved.

To prove (2), recall that a probability measure is unimodal if its density with respect to Lebesgue measure has
a unique local maximum. $\et{\lambda_t}$ extends continuously to $\mathrm{T}$, we thus have that
  \begin{equation}\label{unimodal}
     \frac{d\lambda_t(e^{-i\theta})}{d\theta}=\frac{1}{2\pi}\frac{1-|\et{\lambda_t}(e^{i\theta})|^2}{|1-\et{\lambda_t}(e^{i\theta})|^2}.
  \end{equation}
We first prove the case when $0<t<4$.
From $\eta_{\lambda_t}(\Phi_t(z))=z$ for $z\in\overline{\Omega_t}$ and $\eta_{\lambda_t}(1)=x(t)$, to prove $\lambda_t$ is unimodal, by the boundary correspondence, it is enough to show that the function $f_t$ of $u$ defined by
               \begin{equation}\nonumber
                f_t(u):=  \frac{1-|\gamma_t(u)|^2}{|1-\gamma_t(u)|^2},
               \end{equation}
is a decreasing function on $[0,1]$ and is maximal at $0$. Since $\gamma_t(u)\in\partial\Omega_t$, we have that $|\Phi_t(\gamma_t(u))|=1$. In other words, we have
    \begin{equation}\label{eq:5.4}
       |\gamma_t(u)|\exp\left(\frac{t}{2}f_t(u)\right)= |\gamma_t(u)|\exp\left(\frac{t}{2}\frac{1-|\gamma_t(u)|^2}{|1-\gamma_t(u)|^2} \right)=1.
    \end{equation}
As we shown in Lemma \ref{symmetryMC} that the function $|\gamma_t(u)|$ is an increasing function of $u$, from (\ref{eq:5.4}), we deduce that
$f_t$ is a decreasing function of $u$ and $\max\{f_t\}=f_t(0)$. By the symmetric property of the function $\Phi_t$ in Lemma \ref{symmetryMC}, the density function is symmetric with respect to $x$-axis as well. Thus the density of $\lambda_t$ has only one local maximum at $\Phi_t(\gamma_t(0))=\Phi_t(x_1(t))=1$.

The proof for the case $t\geq 4$ is similar. 
In this case $A_t=\mathrm{T}$ and $\max\{f_t\}=f_t(0)$ and $\min\{f_t\}=f_t(1)$.
Part (3) is a consequence of Corollary \ref{cor:3.27}.
\end{proof}

\begin{rmk}
 \emph{From the proof of Theorem \ref{supportU}, we see that, for $t<4$,
     \begin{equation}\nonumber
        \arg(a_1(t))=\theta_1(t)+\sin(\theta_1(t))=\frac{1}{2}\sqrt{t(4-t)}+\arccos\left(1-\frac{t}{2} \right),
     \end{equation}
which implies a known result in \cite{BianeJFA}. That is
     \begin{equation}\nonumber
        A_t=\left\{e^{i\theta}: -\frac{1}{2}\sqrt{t(4-t)}-\arccos\left(1-\frac{t}{2} \right)
        \leq \theta \leq \frac{1}{2}\sqrt{t(4-t)}+\arccos\left(1-\frac{t}{2} \right)\right \}.
     \end{equation}
      }
\end{rmk}
\subsection{The positive half line case}
Let $\lambda\in\Mrealplus$ be the analogue of the normal distribution such that
$\Sigma_{\lambda}(z)=\exp\left(\frac{z+1}{z-1} \right)$.

We restate Proposition 6.14 in \cite{BV1993} in terms of $\eta$ and $\Sigma$
transforms as follows.
\begin{lemma}\label{leftInv}
 Let $\mu$ be a $\boxtimes$-infinitely divisible measure on $\mathbb{R}^+$,
 and set $\Phi_{\mu}(z):=z\sig{\mu}(z) $.
  \begin{enumerate}[$(1)$]
    \item We have $\Phi_{\mu}\left( \et{\mu}(z)\right)=z$ for every $z\in\mathbb{C}^+$.
    \item The set $\{\et{\mu}(z):z\in\mathbb{C}^+\}=\Omega$, where $\Omega$
          is the component of the set
          $\{z\in\mathbb{C}^+:\Im(\Phi_{\mu}(z))>0 \}$
          whose boundary contains the left half line $(-\infty,0)$.
          Moreover, $\et{\mu}(\Phi_{\mu}(z))=z$ for $z\in\Omega$.
  \end{enumerate}
\end{lemma}
We denote $\Phi_{\lambda}(z)=z\exp\left(\frac{z+1}{z-1} \right)$.
The following lemma is elementary.
\begin{lemma}\label{derivative}
  $\Phi_{\lambda}$ has zero of order one at $2-\sqrt{3}$ and $2+\sqrt{3}$,
  and $\Phi_{\lambda}$ has an essential sigularity at $1$. These are the only
  zeros and sigularities of $\Phi_{\lambda}$.
\end{lemma}

\begin{thm}\label{mainthmG}
 The measure $\lambda$ has no sigular part. The support of this measure
 is the closure of its interior, and this interior has only one connected component.
\end{thm}
\begin{proof}
By Theorem 7.5 in \cite{BV1992}, the measure $\lambda$
has compact support on $\mathbb{R}^+$.

Let $\Omega$ be the component of $\{z\in\mathbb{C}^+:\Im(\Phi_{\lambda}(z))>0\}$
whose boundary contains $(-\infty,0)$.
By Lemma \ref{leftInv}, $\et{\lambda}:\mathbb{C}^+\rightarrow\Omega$
is a conformal map
and $\Phi_{\lambda}$ is its inverse map, thus $\Omega$ is simply connected. By Lemma \ref{derivative},
$\partial \Omega$ is locally analytic.  A general theorem in complex analysis tells us
that $\et{\lambda}$ extends continuously to $\mathbb{C}^+\cup\mathbb{R}$
and it establishes a homeomorphism between the real axis
and $\partial\Omega$. We continue to denote by $\et{\lambda}$
and $\Phi_{\lambda}$ their extensions.

We claim that
 \begin{equation}\nonumber
   \partial\Omega=\left(-\infty,2-\sqrt{3}\right]
    \cup\left[2+\sqrt{3},+\infty\right)\cup\mathcal{L},
 \end{equation}
 where $\mathcal{L}$ is an analytic and open curve in $\mathbb{C}^+$ with endpoints
  $2-\sqrt{3}$ and $2+\sqrt{3}$.
We denote $\gamma(t)=\et{\lambda}(t),t\in\mathbb{R}$ be a parametrization of $\partial\Omega$.
Set $t_1=\Phi_{\lambda}(2-\sqrt{3})>0$ and $t_2=\Phi_{\lambda}(2+\sqrt{3})>0$,
then $\et{\lambda}(t_1)=2-\sqrt{3}$ and $\et{\lambda}(t_2)=2+\sqrt{3}$,
and $\mathcal{L}=\{\gamma(t)\}_{t_1<t<t_2}$.
Note that
\begin{enumerate}[$(1)$]
  \item $(-\infty,0)\subset \partial\Omega$,
  \item $\Phi_{\lambda}'(x)>0$ for all $x\in (-\infty,2-\sqrt{3})$.
\end{enumerate}
From this we deduce that $(-\infty,2-\sqrt{3})\subset \partial\Omega$.
Lemma \ref{derivative} tells us $\Phi_{\lambda}$
has a zero of order one at $2-\sqrt{3}$,
therefore $\partial\Omega$ has a corner of opening $\pi/2$ at $2-\sqrt{3}$.
Note that $\Phi_{\lambda}'(x)>0$ for all $x\in (2+\sqrt{3},+\infty)$,
thus $(2+\sqrt{3},+\infty)\subset\partial\Omega$,
and $\partial\Omega$ has a corner of opening $\pi/2$ at $2+\sqrt{3}$.

It remains to prove that $\mathcal{L}\cap \mathbb{R}=\emptyset $.
First we show $1 \notin \mathcal{L}$. Suppose that is the case,
and suppose $\gamma(t_0)=1$ where $t_1<t_0<t_2$,
by continuity, we have
 \begin{equation}\nonumber
   \gamma(t)\exp\left(\frac{\gamma(t)+1}{\gamma(t)-1} \right) =
    \Phi_{\lambda}(\gamma(t))=\Phi_{\lambda}(\et{\lambda}(t))=t
 \end{equation}
for all $t\in\mathbb{R}$. Therefore in a small neighborhood
of $t_0$, we have
  \begin{equation}\nonumber
     \frac{\gamma(t)+1}{\gamma(t)-1}=\ln\left(\frac{t}{\gamma(t)}\right).
  \end{equation}
The left hand of the above equation blows up, while the right hand side
is bounded. This contradiction tells us that $1\notin \mathcal{L}$.
Now suppose $\mathcal{L}$ touches the real axis at
$x_0\in (2-\sqrt{3},1)\cup(1,2+\sqrt{3})$. Since $\Omega$ is
connected, it is not hard to see that $x_0$ must be a critical point of $\Phi_{\lambda}$.
This is not possible by Lemma \ref{derivative}. We therefore
proved that $\mathcal{L}\subset \mathbb{C}^+$ and the claim.

From the definitions of the Cauchy transform and $\eta$-transform,
one can easily check that
   \begin{equation}\nonumber
     \g{\lambda}\left(\frac{1}{z}\right)=\frac{z}{1-\et{\lambda}(z)}.
   \end{equation}
From the above equation we know that $\g{\lambda}$
extends to be a continuous function on $\mathbb{C}\cup\mathbb{R}$,
and $\{x\in\mathbb{R}:\Im(\g{\mu}(x))>0\}=(1/t_2,1/t_1)$.
By the Stieltjes inverse formula, we deduce that
the support of $\lambda$ is $(1/t_2,1/t_1)$. From
the analyticity of the curve $\mathcal{L}\subset\mathbb{C}^+$,
we conclude that $\lambda$ has positive and analytic density
in the interior of its support.
\end{proof}

We are interested in the level curves of the function
\begin{equation}\label{eq:5.11}
   f(r, \theta)=\theta- \frac{2r \sin\theta}{1-2r \cos\theta +r^2}=\arg(\Phi_{\lambda}(z)),
\end{equation}
where $z=r^{i\theta}\in\mathbb{C}^+$.
For $t\leq 0$, set $\gamma_t=\{z=re^{i\theta}\in\mathbb{C}^+:f(r,\theta)=t\}$.
\begin{prop}
\begin{enumerate}[(A)]
  \item $\gamma_0$ is a simple open curve with endpoints $2-\sqrt{3}, 2+\sqrt{3}$ and $\gamma_0=\mathcal{L}$.
  \item $\gamma_t$ is a simple open curve which starts at $z=1$ and ends at $z=1$ as well for all $t<0$.
\end{enumerate}
Denote by $\Omega_0$ the open domain bounded $\gamma_0\cup[2-\sqrt{3},2+\sqrt{3}]$. For all $t<0$,
Denote by $\Omega_t$ the open domain bounded $\gamma_t\cup\{1\}$.
\begin{enumerate}[(C)]
  \item For $t_1<t_2\leq 0$, we have that $\Omega_{t_1}\subset\Omega_{t_2}$;
and for all $t_0\leq 0$, $\Omega_{t_0}=\cup_{t<t_0}\Omega_t$.
\end{enumerate}
\end{prop}
\begin{proof}
Given $\theta\in (0,\pi)$, we define a function of $r$ by
 $f_{\theta}(r)=f(r,\theta)$ for $r\in (0,+\infty)$. We first note that
   $f(r,\theta)<\theta<\pi$ and observe that
         \begin{equation}\nonumber
            \lim_{r\rightarrow +\infty}f_{\theta}(r)=\theta.
         \end{equation}
We thus have that $\{z=re^{i\theta}:f(r,\theta)>0,\, 0<\theta<\pi\}\subset \Phi^{-1}(\mathbb{C}^+)$.

Given $\theta\in(0,\pi)$ and $t\leq 0$, the equation $f(r,\theta)=t$ is equivalent to the quadratic equation
      \begin{equation}\label{eq:5.13}
        h_{\theta}(r):= (\theta-t)r^2-( 2(\theta-t)\cos\theta + 2\sin\theta   )r + \theta-t=0
      \end{equation}
with discriminant
$d(\theta,t)=[2(\theta-t)\cos\theta+2\sin\theta]^2-4(\theta-t)^2$. We then
rewrite $d(\theta,t)$ as follows.
     \begin{equation}\label{eq:5.10}
         d(\theta,t)=4(1-\cos^2(\theta))\left[\frac{\sin\theta}{1+\cos\theta}+\theta-t \right].
            \left[\frac{\sin\theta}{1-\cos\theta}-\theta+t \right],
     \end{equation}
 We observe that the first two factors in (\ref{eq:5.10}) are
never zero for $\theta\in(0,\pi)$, thus only the last factor in (\ref{eq:5.10}) matters to determine the sign of $d(\theta, t)$. We consider the function $k$ by $k(\theta)=\sin\theta/(1-\cos\theta)-\theta$ for $\theta\in(0,\pi)$, and calculate
       \begin{equation}\label{eq:5.12}
           k'(\theta)=\frac{1}{\cos\theta-1}-1<0,
       \end{equation}
which implies that $k$ is a decreasing function of $\theta$. For $t\leq 0$, we now set $d_t(\theta):=d(\theta,t)$.
We then deduce that $d_t(\theta)=0$ has exactly one solution, which we denote by $\theta_t$, and $d_t(\theta)>0$ if and only if $0<\theta<\theta_t$. Therefore, the half line $r=\theta$ intersects with $\gamma_t$ at two points if and only if $0<\theta<\theta_t$ and the half line $r=\theta_t$ is tangent to $\gamma_t$. Moreover, $\theta_{t_1}<\theta_{t_2}$ if $t_1<t_2\leq 0$.

For the solutions of the equation $f(r,\theta)=0$, one can check as $\theta\rightarrow 0$, $r$ satisfying the equation
$r^2-4r+1$. Given $t<0$, for the solutions of the equation $f(r,\theta)=t$, we can easily see that $r$ tend to $1$ as $\theta\rightarrow 0$. Now (A) and (B) follow from this observation.

Given $\theta\in(0,\pi)$, from (\ref{eq:5.11}), we see that the function $f_{\theta}(r)$ defined by $f_{\theta}(r)=f(r,\theta)$ has exactly one local minimum at $r=1$. $f_{\theta}(r)$ is a decreasing function of $r$ on $(0,1)$ and an increasing function of $r$ on $(1,\infty)$. Therefore, if the half line $r=\theta$ intersects with $\gamma_t$ at two points, then one of them is inside the unit circle of $\mathbb{C}$ and the other one is outside the unit circle. We conclude that (C) is valid.
\end{proof}

It is interesting to compare the following result with
Proposition \ref{leftInvreal} and Lemma \ref{infMstar}.
\begin{cor}
We have that $\Phi_{\lambda}^{-1}(\mathbb{C}^+) = \Omega\cup_{k=1}^{\infty}(\Omega_{(2k-1)\pi}\backslash\Omega_{(2k-2)\pi})$. Moreover, $\Omega$ and $\Omega_{(2k-1)\pi}\backslash\Omega_{(2k-2)\pi}$ ($k=1,2,\cdots$) are all connected components of $\Phi_{\lambda}$. In particular, $\Phi_{\lambda}^{-1}(\mathbb{C}^+)$ has infinitely many connected components.
\end{cor}

We would like to point out that for $z=re^{i\theta}\in\mathcal{L}=\gamma_0$, the curve $\mathcal{L}$ can be parametrized by $r$. Noticing (\ref{eq:5.13}) and (\ref{eq:5.10}), we first observe the following equivalence relations:
 \begin{equation}\label{delta0}
 d(\theta,0)=0  \Leftrightarrow  \theta\cos\theta+\sin\theta=\theta \Leftrightarrow r=1.
 \end{equation}
By (\ref{eq:5.12}), we see that (\ref{delta0}) has exactly one solution $\theta_0$ for $\theta\in(0,\pi)$.
By differentiating the equation $f(r,\theta)=0$, we obtain that
   \begin{equation}\label{dtheta}
      \frac{d\theta}{dr}=\frac{2\theta\cos\theta+2\sin\theta-2\theta r}{r^2+2\theta\sin\theta-4\cos\theta +1}.
   \end{equation}
 Thus, $d\theta/dr=0$ if and only if $r=(\theta\cos\theta+\sin\theta)/\theta$. Fix $\theta$, the equation $f_{\theta}(r)=0$ is equivalent to the quadratic equation $\theta r^2- (2\theta\cos\theta +2\sin\theta)r+\theta=0$, from which we deduce that $r=(\theta\cos\theta+\sin\theta)/\theta$ if and only if $d(r,0)=0$. From (\ref{dtheta}) and continuity of $d\theta/dr$, we see that $d\theta/dr>0$ for $0<\theta<\theta_0, r<1$ and
 $d\theta/dr<0$ for $0<\theta<\theta_0, r>1$. Therefore, for the solutions of the equation $f(r,\theta)=0$, $\theta$ is a function of $r$ and the curve $\mathcal{L}$ can be parametrized by $r$.

We denote by $g$ the density function of $\lambda$. From the equation $\g{\lambda}(1/x)=x/(1-\et{\lambda}(x))$, we obtain the following formula for the density function of $\lambda$.
\begin{prop}
 Given $z=re^{i\theta}\in\gamma_0=\mathcal{L}$, we have
      \begin{equation}\nonumber
         g(1/x)=\theta\Phi_{\lambda}(z)=r\theta\exp\left(\frac{r^2-1}{1-2r\cos\theta+r^2}\right),
      \end{equation}
where $x=\Phi_{\lambda}(z)$.
\end{prop}
 \begin{figure}[H]
     \includegraphics[scale=0.55]{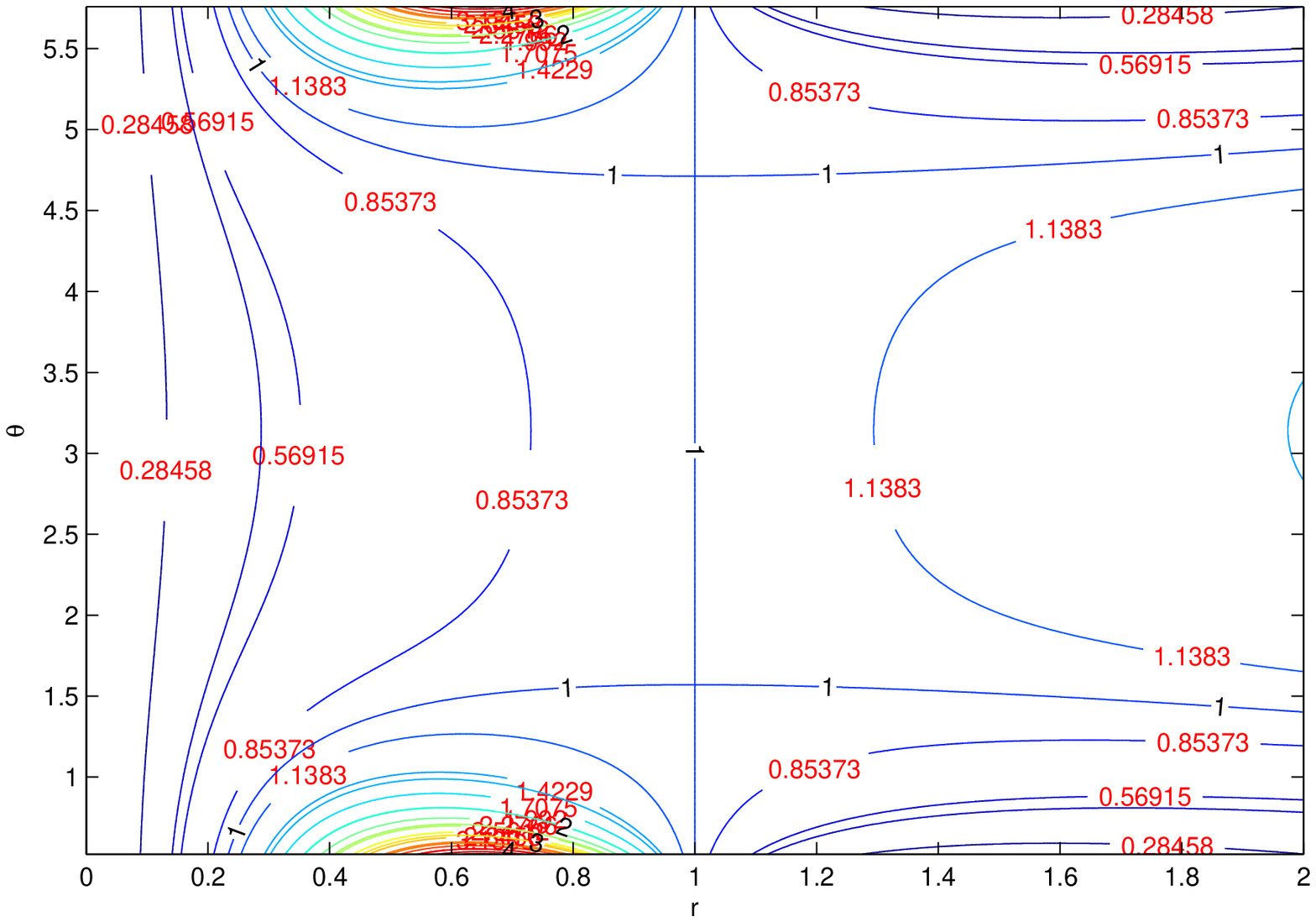}
    \vspace{0pt}
    \caption{Level curves of $g_2(r,\theta)=|\Phi_2(re^{i\theta})|$. The vertical axis indicates
     $\theta$, and the horizontal axis indicates $r$.}
    \label{fig1}
   \end{figure}
    \begin{figure}[H]
    \includegraphics[scale=0.55]{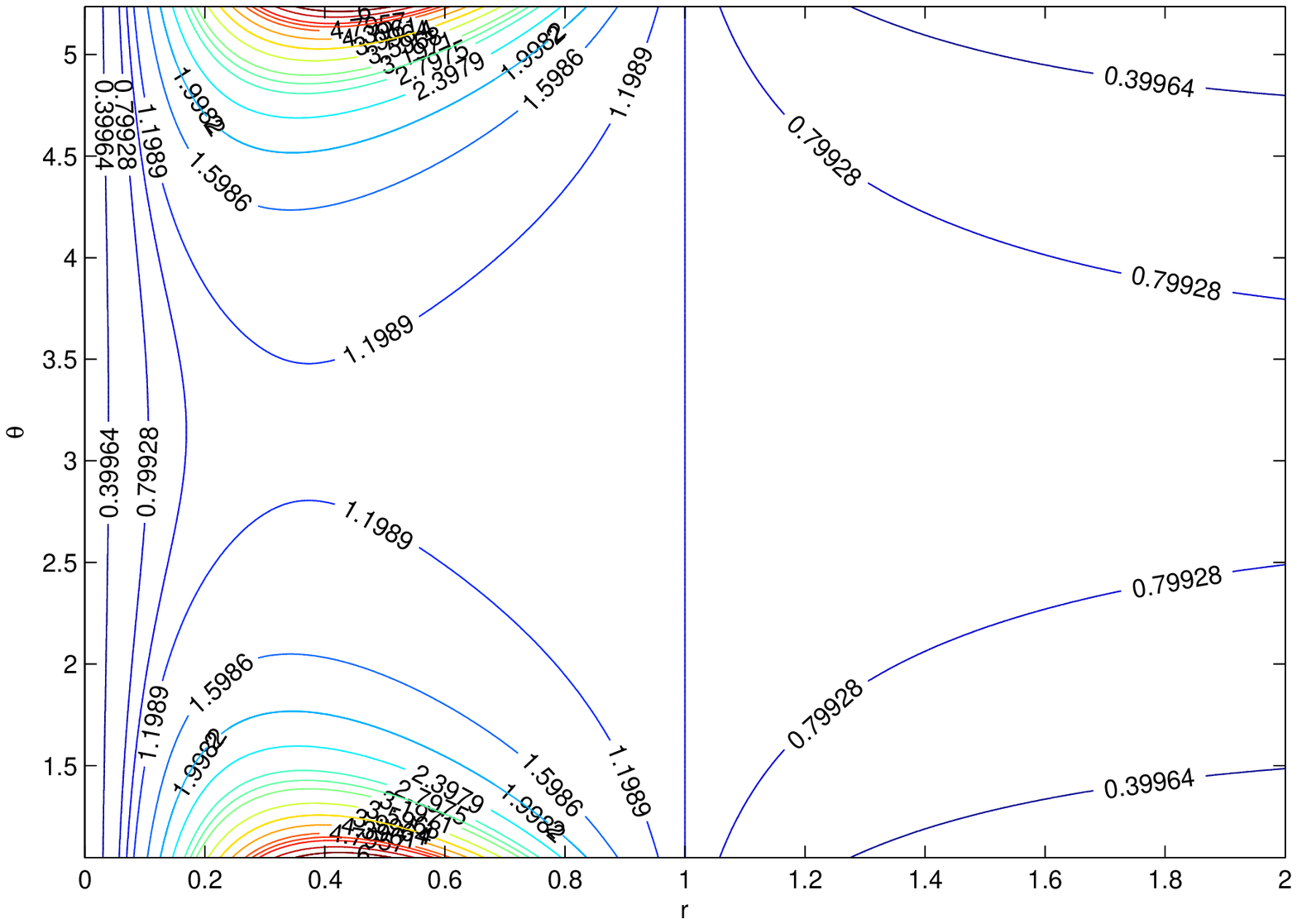}
    \vspace{0pt}
    \caption{Level curves of $g_5(r,\theta)=|\Phi_5(re^{i\theta})|$. The vertical axis indicates
     $\theta$, and the horizontal axis indicates $r$.}
    \label{fig2}
   \end{figure}
\begin{figure}[H]
 \vspace{-10pt}
   \includegraphics[scale=0.6]{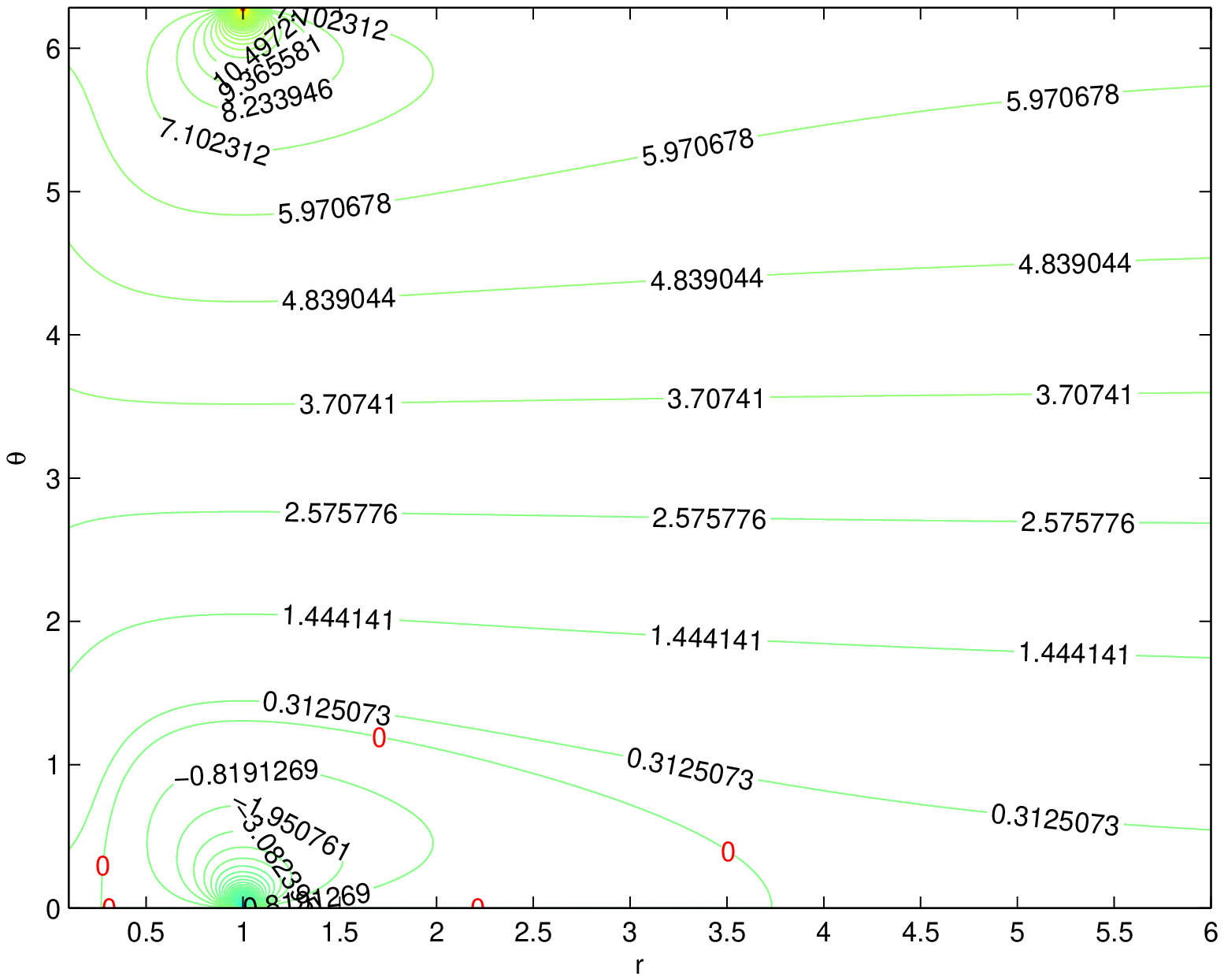}
    \vspace{0pt}
    \caption{Level curves of $f(r,\theta)=\arg(\Phi_{\lambda}(re^{i\theta}))$. The vertical axis indicates
     $\theta$, and the horizontal axis indicates $r$.}
    \label{fig3}
   \end{figure}

\section*{Acknowledgements}
I am grateful to my advisor, Professor Hari Bercovici, for encouragement and many helpful discussions which make this work available. I thank Octavio Arizmendi for informing me his paper \cite{Ariz2012} and Professor Michael Anshelevich for his interest. I would also like to thank Professor Eric Bedford for several helpful conversations and Aimin Huang for his careful reading of a preliminary version of this paper.

\bibliography{clipboard}

\providecommand{\bysame}{\leavevmode\hbox to3em{\hrulefill}\thinspace}
\providecommand{\MR}{\relax\ifhmode\unskip\space\fi MR }
\providecommand{\MRhref}[2]{%
  \href{http://www.ams.org/mathscinet-getitem?mr=#1}{#2}
}
\providecommand{\href}[2]{#2}
\begin{thebibliography}{10}

\bibitem{Anshelevich}
Michael Anshelevich, \emph{Free evolution on algebras with two states}, J.
  Reine Angew. Math. \textbf{638} (2010), 75--101. \MR{2595336 (2012d:46155)}

\bibitem{AnshG2011}
\bysame, \emph{Generators of some non-commutative stochastic process},
  arXiv:1104.1381v2 math.OA (2011).

\bibitem{AnshetwostateBM}
\bysame, \emph{Two-state free {B}rownian motions}, J. Funct. Anal. \textbf{260}
  (2011), no.~2, 541--565. \MR{2737412 (2012c:46147)}

\bibitem{Anshelevich2012}
\bysame, \emph{Free evolution on algebras with two states {II}},
  arXiv:1204.0289v1 math.OA (2012).

\bibitem{Ariz2012}
Octavio Arizmendi, \emph{$k$-divisible random variables in free probability},
  arXiv:1203.4780v1 math.PR (2012).

\bibitem{AHasebe2012}
Octavio Arizmendi and Takahiro Hasebe, \emph{Semigroups related to additive and
  multiplicative, free and boolean convolutions}, arXiv:1105.3344v3 math.PR
  (2012).

\bibitem{BB2004}
Serban~T. Belinschi and Hari Bercovici, \emph{Atoms and regularity for measures
  in a partially defined free convolution semigroup}, Math. Z. \textbf{248}
  (2004), no.~4, 665--674. \MR{2103535 (2006i:46095)}

\bibitem{BB2005}
\bysame, \emph{Partially defined semigroups relative to multiplicative free
  convolution}, Int. Math. Res. Not. (2005), no.~2, 65--101. \MR{2128863
  (2006f:46061)}

\bibitem{BB2007new}
\bysame, \emph{A new approach to subordination results in free probability}, J.
  Anal. Math. \textbf{101} (2007), 357--365. \MR{2346550 (2008i:46059)}

\bibitem{BN2008}
Serban~T. Belinschi and Alexandru Nica, \emph{{$\eta$}-series and a {B}oolean
  {B}ercovici-{P}ata bijection for bounded {$k$}-tuples}, Adv. Math.
  \textbf{217} (2008), no.~1, 1--41. \MR{2357321 (2009c:46088)}

\bibitem{BN2008Ind}
\bysame, \emph{On a remarkable semigroup of homomorphisms with respect to free
  multiplicative convolution}, Indiana Univ. Math. J. \textbf{57} (2008),
  no.~4, 1679--1713. \MR{2440877 (2009f:46087)}

\bibitem{BN2009}
\bysame, \emph{Free {B}rownian motion and evolution towards
  {$\boxplus$}-infinite divisibility for {$k$}-tuples}, Internat. J. Math.
  \textbf{20} (2009), no.~3, 309--338. \MR{2500073 (2010g:46108)}

\bibitem{Bhari2006}
Hari Bercovici, \emph{On {B}oolean convolutions}, Operator theory 20, Theta
  Ser. Adv. Math., vol.~6, Theta, Bucharest, 2006, pp.~7--13. \MR{2276927
  (2007m:46105)}

\bibitem{BPata1999}
Hari Bercovici and Vittorino Pata, \emph{Stable laws and domains of attraction
  in free probability theory}, Ann. of Math. (2) \textbf{149} (1999), no.~3,
  1023--1060, With an appendix by Philippe Biane. \MR{1709310 (2000i:46061)}

\bibitem{BV1992}
Hari Bercovici and Dan Voiculescu, \emph{L\'evy-{H}in\v cin type theorems for
  multiplicative and additive free convolution}, Pacific J. Math. \textbf{153}
  (1992), no.~2, 217--248. \MR{1151559 (93k:46052)}

\bibitem{BV1993}
\bysame, \emph{Free convolution of measures with unbounded support}, Indiana
  Univ. Math. J. \textbf{42} (1993), no.~3, 733--773. \MR{1254116 (95c:46109)}

\bibitem{BianeBM}
Philippe Biane, \emph{Free {B}rownian motion, free stochastic calculus and
  random matrices},  \textbf{12} (1997), 1--19. \MR{1426833 (97m:46104)}

\bibitem{Biane1997}
\bysame, \emph{On the free convolution with a semi-circular distribution},
  Indiana Univ. Math. J. \textbf{46} (1997), no.~3, 705--718. \MR{1488333
  (99e:46084)}

\bibitem{BianeJFA}
\bysame, \emph{Segal-{B}argmann transform, functional calculus on matrix spaces
  and the theory of semi-circular and circular systems}, J. Funct. Anal.
  \textbf{144} (1997), no.~1, 232--286. \MR{1430721 (97k:22011)}

\bibitem{Biane1998}
\bysame, \emph{Processes with free increments}, Math. Z. \textbf{227} (1998),
  no.~1, 143--174. \MR{1605393 (99e:46085)}

\bibitem{arithmetic}
Gennadii~P. Chistyakov and Friedrich Götze, \emph{The arithmetic of
  distributions in free probability theory}, Central European Journal of
  Mathematics \textbf{9} (2011), 997--1050, 10.2478/s11533-011-0049-4.

\bibitem{CG2011}
\bysame, \emph{Asymptotic expansion in the \uppercase{CLT} in free
  probability}, arXiv:1109.4844v2 [math.PR] (2011).

\bibitem{DH2011}
Nizar Demni and Taoufik Hmidi, \emph{Spectral distribution of the free unitary
  brownian motion: another approach}, arXiv: 1103.4693 math.OA (2011).

\bibitem{Franz}
Uwe Franz, \emph{Boolean convolution of probability measures on the unit
  circle}, Analyse et probabilit\'es, S\'emin. Congr., vol.~16, Soc. Math.
  France, Paris, 2008, pp.~83--94. \MR{2599263 (2011c:46136)}

\bibitem{Hille}
Einar Hille, \emph{Analytic function theory. {V}ol. 1}, Introduction to Higher
  Mathematics, Ginn and Company, Boston, 1959. \MR{0107692 (21 \#6415)}

\bibitem{LR2007}
Romuald Lenczewski, \emph{Decompositions of the free additive convolution}, J.
  Funct. Anal. \textbf{246} (2007), no.~2, 330--365. \MR{2321046 (2008d:28009)}

\bibitem{Maassen}
Hans Maassen, \emph{Addition of freely independent random variables}, J. Funct.
  Anal. \textbf{106} (1992), no.~2, 409--438. \MR{1165862 (94g:46069)}

\bibitem{Nica2009}
Alexandru Nica, \emph{Multi-variable subordination distributions for free
  additive convolution}, J. Funct. Anal. \textbf{257} (2009), no.~2, 428--463.
  \MR{2527024 (2010j:46121)}

\bibitem{NicaS1996}
Alexandru Nica and Roland Speicher, \emph{On the multiplication of free
  {$N$}-tuples of noncommutative random variables}, Amer. J. Math. \textbf{118}
  (1996), no.~4, 799--837. \MR{1400060 (98i:46069)}

\bibitem{RS2007}
N.~Raj Rao and Roland Speicher, \emph{Multiplication of free random variables
  and the {$S$}-transform: the case of vanishing mean}, Electron. Comm. Probab.
  \textbf{12} (2007), 248--258. \MR{2335895 (2008f:46082)}

\bibitem{SpeicherW1997}
Roland Speicher and Reza Woroudi, \emph{Boolean convolution}, Free probability
  theory ({W}aterloo, {ON}, 1995), Fields Inst. Commun., vol.~12, Amer. Math.
  Soc., Providence, RI, 1997, pp.~267--279. \MR{1426845 (98b:46084)}

\bibitem{DVV1986}
Dan Voiculescu, \emph{Addition of certain noncommuting random variables}, J.
  Funct. Anal. \textbf{66} (1986), no.~3, 323--346. \MR{839105 (87j:46122)}

\bibitem{Wang}
Jiun-Chau Wang, \emph{Limit laws for {B}oolean convolutions}, Pacific J. Math.
  \textbf{237} (2008), no.~2, 349--371. \MR{2421126 (2009h:46128)}

\bibitem{Wang2010}
\bysame, \emph{Local limit theorems in free probability theory}, Ann. Probab.
  \textbf{38} (2010), no.~4, 1492--1506. \MR{2663634 (2011i:46081)}

\end{thebibliography}
\bibliographystyle{amsplain}

\end{document}